\documentclass[reqno]{amsart}

\usepackage{tikz-cd}

\usepackage{fullpage,dsfont}

\usepackage{hyperref} 

\usepackage{parskip}
\makeatletter 
\def\thm@space@setup{%
 \thm@preskip=\parskip \thm@postskip=0pt
}
\def\th@remark{%
  \thm@headfont{\itshape}%
  \normalfont 
  \thm@preskip\parskip \thm@postskip=0pt
}
\makeatother

\usepackage[nobysame,alphabetic,initials,msc-links]{amsrefs}

\usepackage[final]{pdfpages}

\usepackage{multirow}

\usepackage{array}

\DefineSimpleKey{bib}{how}
\DefineSimpleKey{bib}{mrclass}
\DefineSimpleKey{bib}{mrnumber}
\DefineSimpleKey{bib}{fjournal}
\DefineSimpleKey{bib}{mrreviewer}

\renewcommand{\PrintDOI}[1]{%
  \href{http://dx.doi.org/#1}{{\tt DOI:#1}}%
}
\renewcommand{\eprint}[1]{#1}
\BibSpec{book}{%
    +{}  {\PrintPrimary}                {transition}
    +{.} { \PrintDate}                  {date}
    +{.} { \textit}                     {title}
    +{.} { }                            {part}
    +{:} { \textit}                     {subtitle}
    +{,} { \PrintEdition}               {edition}
    +{}  { \PrintEditorsB}              {editor}
    +{,} { \PrintTranslatorsC}          {translator}
    +{,} { \PrintContributions}         {contribution}
    +{,} { }                            {series}
    +{,} { \voltext}                    {volume}
    +{,} { }                            {publisher}
    +{,} { }                            {organization}
    +{,} { }                            {address}
    +{,} { }                            {status}
    +{,} { \PrintDOI}                   {doi}
    +{,} { \PrintISBNs}                 {isbn}
    +{}  { \parenthesize}               {language}
    +{}  { \PrintTranslation}           {translation}
    +{;} { \PrintReprint}               {reprint}
    +{.} { }                            {note}
    +{.} {}                             {transition}
    +{}  {\SentenceSpace \PrintReviews} {review}
}
\BibSpec{article}{%
    +{}  {\PrintAuthors}                {author}
    +{,} { \textit}                     {title}
    +{.} { }                            {part}
    +{:} { \textit}                     {subtitle}
    +{,} { \PrintContributions}         {contribution}
    +{.} { \PrintPartials}              {partial}
    +{,} { }                            {journal}
    +{}  { \textbf}                     {volume}
    +{}  { \PrintDatePV}                {date}
    +{,} { \issuetext}                  {number}
    +{,} { \eprintpages}                {pages}
    +{,} { }                            {status}
    +{,} { \PrintDOI}                   {doi}
    +{,} { \eprint}        {eprint}
    +{}  { \parenthesize}               {language}
    +{}  { \PrintTranslation}           {translation}
    +{;} { \PrintReprint}               {reprint}
    +{.} { }                            {note}
    +{.} {}                             {transition}
    +{}  {\SentenceSpace \PrintReviews} {review}
}
\BibSpec{collection.article}{%
    +{}  {\PrintAuthors}                {author}
    +{,} { \textit}                     {title}
    +{.} { }                            {part}
    +{:} { \textit}                     {subtitle}
    +{,} { \PrintContributions}         {contribution}
    +{,} { \PrintConference}            {conference}
    +{}  {\PrintBook}                   {book}
    +{,} { }                            {booktitle}
    +{,} { \PrintDateB}                 {date}
    +{,} { pp.~}                        {pages}
    +{,} { }                            {publisher}
    +{,} { }                            {organization}
    +{,} { }                            {address}
    +{,} { }                            {status}
    +{,} { \PrintDOI}                   {doi}
    +{,} { \eprint}        {eprint}
    +{}  { \parenthesize}               {language}
    +{}  { \PrintTranslation}           {translation}
    +{;} { \PrintReprint}               {reprint}
    +{.} { }                            {note}
    +{.} {}                             {transition}
    +{}  {\SentenceSpace \PrintReviews} {review}
}
\BibSpec{misc}{%
  +{}{\PrintAuthors}  {author}
  +{,}{ \textit}      {title}
  +{.}{ }             {how}
  +{}{ \parenthesize} {date}
  +{,} { available at \eprint}        {eprint}
  +{,}{ available at \url}{url}
  +{,}{ }             {note}
  +{.}{}              {transition}
}
\usepackage{amssymb, amsfonts, amsxtra, amsmath}
\usepackage{mathrsfs}
\usepackage{mathdots}
\usepackage{wasysym}
\usepackage[all]{xy}
\usepackage{bbm}
\usepackage{calc}
\usepackage{accents}

\usepackage[bbgreekl]{mathbbol}			

\usepackage{stackengine}

\numberwithin{equation}{section}

\DeclareSymbolFontAlphabet{\mathbb}{AMSb}	
\DeclareSymbolFontAlphabet{\mathbbl}{bbold}	

\newtheorem{Theorem}{Theorem}[section]
\newtheorem*{Theorem*}{Theorem}
\newtheorem{Def}[Theorem]{Definition}
\newtheorem*{Def*}{Definition}
\newtheorem{Lem}[Theorem]{Lemma}
\newtheorem{Prop}[Theorem]{Proposition}

\newtheorem{Rem}[Theorem]{Remark}

\newcommand\bp{\begin{proof}}
\newcommand\ep{\end{proof}}

\mathchardef\mhyph="2D

\DeclareMathOperator{\Hom}{\mathrm{Hom}}
\DeclareMathOperator{\id}{\mathrm{id}}

\DeclareMathOperator{\Mor}{\mathrm{Mor}}

\DeclareMathOperator{\Rep}{\mathrm{Rep}}
\DeclareMathOperator{\Irr}{\mathrm{Irr}}

\DeclareMathOperator{\rd}{\mathrm{d}\!}

\DeclareMathOperator{\coarse}{\mathrm{coarse}}

\DeclareMathOperator{\fd}{\mathrm{fd}}

\newcommand{\op}{\mathrm{op}}

\newcommand{\msN}{\mathscr{N}}

\newcommand{\msU}{\mathscr{U}}

\newcommand{\mcA}{\mathcal{A}}

\newcommand{\mcB}{\mathcal{B}}
\newcommand{\mcC}{\mathcal{C}}
\newcommand{\mcD}{\mathcal{D}}

\newcommand{\Hsp}{\mathcal{H}}
\newcommand{\Gsp}{\mathcal{G}}

\newcommand{\mcK}{\mathcal{K}}

\newcommand{\mcM}{\mathcal{M}}
\newcommand{\mcO}{\mathcal{O}}

\newcommand{\mcT}{\mathcal{T}}
\newcommand{\mcU}{\mathcal{U}}

\newcommand{\flip}{\mathrm{flip}}
\newcommand{\conj}{\mathrm{conj}}

\newcommand{\C}{\mathbb{C}}
\newcommand{\G}{\mathbb{G}}
\newcommand{\Hh}{\mathbb{H}}

\newcommand{\N}{\mathbb{N}}

\newcommand{\R}{\mathbb{R}}

\newcommand{\opp}{\mathrm{op}}

\newcommand{\trace}{\mathrm{tr}}

\newcommand{\Corr}{\mathrm{Corr}}

\begin{document}

\title{Amenable actions of compact and discrete quantum groups on von Neumann algebras}
\author{K. De Commer and J. De Ro}
\address{Vrije Universiteit Brussel}
\email{kenny.de.commer@vub.be}
\email{joeri.ludo.de.ro@vub.be}

\begin{abstract} Let $\G$ be a compact or a discrete quantum group and $A\subseteq B$ an inclusion of $\sigma$-finite $\G$-dynamical von Neumann algebras. We prove that the $\G$-inclusion $A\subseteq B$ is strongly equivariantly amenable if and only if it is equivariantly amenable, using techniques from the theory of non-commutative $L^p$-spaces. In particular, if $(A, \alpha)$ is a $\G$-dynamical von Neumann algebra with $A$ $\sigma$-finite, the action $\alpha: A \curvearrowleft \G$ is strongly (inner) amenable if and only if the action $\alpha: A \curvearrowleft \G$ is (inner) amenable.  This result can be seen as a dynamical generalization of Tomatsu's result on the amenability/co-amenability duality for discrete/compact quantum groups. We also provide the first explicit examples of amenable discrete quantum groups that act non-amenably on a von Neumann algebra.
\end{abstract}

\maketitle

\section{Introduction and Preliminaries}

The notion of Zimmer amenability of an action of a locally compact group $G$ on a von Neumann algebra $A$ was introduced and studied in the papers \cites{AD79, AD82}. It plays an important role in the study of the von Neumann algebra crossed product $A \rtimes G$. For example, if the von Neumann algebra $A$ is injective and the action $A \curvearrowleft G$ is amenable, then $A\rtimes G$ is an injective von Neumann algebra. Further, the notion of Zimmer amenability provides a bridge to the notion of amenability for $C^*$-dynamical systems \cites{BEW21,OS21,BC22}. For these reasons, it is natural to generalize the notion of Zimmer amenability to the setting of locally compact quantum groups. This was first done and studied in \cite{Moa18} for discrete quantum groups and in \cites{DCDR23, DR24} for general locally compact quantum groups. In this paper, we continue the study of the notion of amenability of actions, but we restrict ourselves to the setting of compact and discrete quantum groups. We will make extensive use of the general theory developed in \cite{DCDR23}.

More concretely, in \cite{DCDR23}, the notion of equivariant correspondence was introduced. Given a locally compact quantum group $\G$ and two $\G$-$W^*$-dynamical systems $(A, \alpha)$ and $(B, \beta)$ (= von Neumann algebras with an action $\alpha: A \rightarrow A\bar{\otimes} L^{\infty}(\G)$ by $\G$), a $\G$-$A$-$B$-correspondence consists of a Hilbert space $\mathcal{H}$ with an $A$-$B$-bimodule structure by unital normal $*$-preserving maps, and equipped with a unitary $\G$-representation which is compatible with the $A$-$B$-bimodule structure \cite{DCDR23}*{Definition 0.4}. This notion was directly inspired by a very similar notion in the setting of compact quantum groups, introduced in \cite{AS21}. We will comment in more detail on the precise connection in Section \ref{SecEwcc}. We will write $\Corr^{\G}(A,B)$ for the 
W$^*$-category of all such $\G$-equivariant $A$-$B$-correspondences.

There is a natural notion of weak containment $\preccurlyeq$ for $\G$-correspondences \cite{DCDR23}*{Definition 3.1}. This notion can be used to formulate several important approximation properties for dynamical von Neumann algebras. For example, if $A\subseteq B$ is a $\G$-equivariant inclusion, we call it \emph{strongly equivariantly amenable} if we have the weak containment
$${ }_A L^2(A)_A\preccurlyeq {}_{A} L^2(B)_{A}$$
of $\G$-$A$-$A$-correspondences \cite{DCDR23}*{Definition 4.1}.

Given a $\G$-$A$-$A$-correspondence $\mathcal{H}$, it is in general a hard problem to detect if we have the weak containment ${}_A L^2(A)_A\preccurlyeq \mathcal{H}$ of $\G$-$A$-$A$-correspondences. The second section of this paper is devoted to this problem in the setting of compact quantum groups. 

More precisely, let $\G$ be a compact quantum group with associated Hopf $*$-algebra $(\mcO(\G),\Delta)$. Given a right $\G$-$W^*$-algebra $(A, \alpha)$, we write 
\begin{equation}\label{EqAlgCoreOth}
\mathcal{A}:= \{a\in A: \alpha(a)\in A \odot \mathcal{O}(\G)\}
\end{equation}
for its algebraic core (with $\odot$ the algebraic tensor product). The action $\alpha: A\curvearrowleft \G$ induces a left action $\alpha_c: \G \curvearrowright \overline{\mathcal{A}}$ where $\overline{\mathcal{A}}$ is the conjugate algebra of $\mathcal{A}$. We then define the \emph{cotensor product} as the $*$-subalgebra
$$\mathcal{A}\square \overline{\mathcal{A}}:= \{z \in \mathcal{A}\odot \overline{\mathcal{A}}: (\alpha \odot \id)(z) = (\id \odot \alpha_c)(z)\}.$$
We will prove that every $\G$-$A$-$A$-correspondence $\mathcal{H}$ canonically induces a unital $*$-representation $\theta_{\square}^{\mathcal{H}}$ of the $*$-algebra $\mathcal{A}\square \overline{\mathcal{A}}$ on a certain associated Hilbert space. We then have the following result:

\textbf{Proposition \ref{weak containment}.}\textit{  Let $\G$ be a compact quantum group, let $A$ be a $\G$-$W^*$-algebra and $\mathcal{H}\in \Corr^\G(A,A)$. The following statements are equivalent:
\begin{enumerate}
    \item $L^2(A)\preccurlyeq \mathcal{H}$ as $\G$-$A$-$A$-correspondences.
    \item $\theta_\square^{L^2(A)}\preccurlyeq \theta_\square^{\mathcal{H}}$, i.e. $\|\theta_\square^{L^2(A)}(x)\| \le \|\theta_\square^{\mathcal{H}}(x)\|$ for all $x\in \mathcal{A}\square \overline{\mathcal{A}}.$
\end{enumerate}}
This significant result provides a more algebraic way to see if a given equivariant correspondence weakly contains the trivial correspondence. It provides the basis for the remainder of the paper.

For a locally compact quantum group $\G$, let us also recall that a $\G$-equivariant normal inclusion $A\subseteq B$ is called \emph{equivariantly amenable} if there exists a (not necessarily normal) $\G$-equivariant ucp conditional expectation $E: B\to A$. 
As the terminology suggests, strong equivariant amenability of a $\G$-equivariant inclusion $A\subseteq B$ implies its equivariant amenability, but the converse is an important open question. In the case of compact quantum groups however, we make use of Proposition \ref{weak containment} to establish the following general result:

\textbf{Theorem \ref{TheoMainAmenab}.} \textit{Let $\G$ be a compact quantum group. Assume $A \subseteq B$ is an equivariant unital normal inclusion of $\sigma$-finite $\G$-dynamical von Neumann algebras. Then the following are equivalent: 
\begin{enumerate}
\item $A \subseteq B$ is equivariantly amenable. 
\item $A \subseteq B$ is strongly equivariantly amenable. 
\end{enumerate}}
In the non-equivariant setting, this was proven in \cite{BMO20}*{Corollary A.2}. We use Proposition \ref{weak containment} and 
 the theory of non-commutative $L^p$-spaces \cite{JP10} to lift this result to the equivariant setting.

 Using duality arguments, we then obtain the same result for discrete quantum groups:

 \textbf{Theorem \ref{discretemain}} \textit{Let $\mathbbl{\Gamma}$ be a discrete quantum group and assume $A \subseteq B$ is an equivariant unital normal inclusion of $\sigma$-finite $\mathbbl{\Gamma}$-dynamical von Neumann algebras. Then the following are equivalent: }
\begin{enumerate}
\item \textit{$A \subseteq B$ is equivariantly amenable. }
\item \textit{$A \subseteq B$ is strongly equivariantly amenable. }
\end{enumerate}
 
Let us now recall from \cite{DCDR23} the following definitions:  for a locally compact quantum group $\G$, a $\G$-$W^*$-algebra $(A, \alpha)$ is said to be 
\begin{itemize}
\item \emph{(strongly) amenable} if the $\G$-inclusion $\alpha(A)\subseteq A \bar{\otimes}L^\infty(\G)$ is (strongly) equivariantly amenable,  and 
\item  \emph{(strongly) inner amenable} if the $\G$-inclusion $\alpha(A)\subseteq A\rtimes_\alpha \G$ is (strongly) equivariantly amenable (where $A\rtimes_{\alpha}\G$ is equipped with the adjoint $\G$-action).
\end{itemize}
If these conditions hold for $A= \C$, we say that $\G$ is (strongly) amenable, resp.\ (strongly) inner amenable. Note that this is consistent with the usual definitions for these notions \cites{Tom06,Cr19}.

As an application of Theorem \ref{TheoMainAmenab}, we then obtain:

\textbf{Theorem \ref{application}} \textit{Let $\G$ be a compact quantum group and let $(A, \alpha)$ be a $\sigma$-finite $\G$-$W^*$-algebra. The following statements hold:}
    \begin{enumerate}
        \item \textit{$(A, \alpha)$ is strongly amenable if and only if $(A, \alpha)$ is amenable.}
        \item \textit{$(A, \alpha)$ is strongly inner amenable if and only if $(A, \alpha)$ is inner amenable.}
    \end{enumerate}

    By duality, we then also obtain:

 \textbf{Theorem \ref{application2}} \textit{Let $\mathbbl{\Gamma}$ be a discrete quantum group and $(A, \alpha)$ a $\sigma$-finite $\mathbbl{\Gamma}$-$W^*$-algebra. The following statements hold:}
    \begin{enumerate}
        \item \textit{$(A, \alpha)$ is strongly amenable if and only if $(A, \alpha)$ is amenable.}
        \item \textit{$(A, \alpha)$ is strongly inner amenable if and only if $(A, \alpha)$ is inner amenable.}
    \end{enumerate}
\textit{In particular, $\mathbbl{\Gamma}$ is strongly inner amenable if and only if $\mathbbl{\Gamma}$ is inner amenable.}

Note that applying Theorem \ref{application2}.(1) to the trivial action $\mathbb{C}\curvearrowleft \mathbbl{\Gamma}$ recovers Tomatsu's result \cite{Tom06}, which asserts that $\check{\mathbbl{\Gamma}}$ is co-amenable if and only if $\mathbbl{\Gamma}$ is amenable. Therefore, Theorem \ref{application2} can be seen as a dynamical version of Tomatsu's result.

In \cite{AD79}*{Proposition 3.6}, it was proven that every action of an amenable locally compact group  on a von Neumann algebra is automatically amenable. An analogue of this statement is no longer true in the (non-Kac) quantum setting:

\textbf{Theorem \ref{TheoMain2}.}\textit{ Let $\mathbb{H}$ be a compact quantum group and $\G := \mathbb{H}^{\op}\times \mathbb{H}$. Consider the action
$$\alpha: L^\infty(\mathbb{H})\to L^\infty(\mathbb{H})\bar{\otimes} L^\infty(\mathbb{\G})= L^\infty(\mathbb{H})^{\bar{\otimes} 3}: z \mapsto \Delta_{\mathbb{H}}^{(2)}(z)_{213}.$$The action $\alpha: L^\infty(\mathbb{H})\curvearrowleft \G$ is (strongly) amenable if and only if $\mathbb{H}$ is of Kac type.}

The proof strategy for this theorem is summarised at the beginning of section 4. We also provide another strategy that allows us to show that the action of $SU_q(2)$ on a non-standard Podleś sphere is not amenable, see Proposition \ref{Podlesamenable}.

Further, note that if $\G$ is a compact quantum group and the action $\alpha: A \curvearrowleft \G$ fails to be amenable, then the dual action
$A\rtimes_\alpha \G \curvearrowleft \check{\G}$
also fails to be amenable. Therefore, from the discussion above, we find explicit examples of amenable discrete quantum groups acting non-amenably on a von Neumann algebra. Note that in \cite{Moa18}*{Theorem 4.7}, it was claimed that every action of an amenable discrete quantum group is amenable. There is no contradiction however, as the proof in \cite{Moa18} only works in the context of unimodular discrete quantum groups, as was remarked in \cite{AK24}.

\textbf{Notation and conventions.} We denote by $\odot$ the algebraic tensor product over $\C$, by $\otimes$ the tensor product between Hilbert spaces and the minimal tensor product between C$^*$-algebras, and by $\bar{\otimes}$ the spatial tensor product between von Neumann algebras. We freely switch between the terminology W$^*$-algebra and von Neuman algebra.

For $\Hsp,\Gsp$ Hilbert spaces, we write $\mcB(\Hsp,\Gsp)$, resp.\ $\mcB(\Hsp)$, for the space of bounded operators from $\Hsp$ to $\Gsp$, resp.\ from $\Hsp$ to $\Hsp$, and $\mcK(\Hsp,\Gsp)$, resp.\ $\mcK(\Hsp)$, for the space of compact operators. 

If $\Hsp$ is a Hilbert space and $\xi,\eta\in \Hsp$, we write $e_{\xi,\eta}\in B(\Hsp)$ and $\omega_{\xi,\eta} \in B(\Hsp)_*$ as defined by 
\[
e_{\xi,\eta}\zeta = \langle \eta,\zeta\rangle \xi,\qquad \omega_{\xi,\eta}(x) = \langle \xi,x\eta\rangle,\qquad x\in B(\Hsp). 
\]
By $[-]^{\tau}$, we denote the $\tau$-closure of the linear span of a subset in a topological vector space $(V,\tau)$. When this is a fixed norm-topology, we abbreviate to $[-]$.  

\subsection{Von Neumann algebras. }If $A$ is a von Neumann algebra, we denote the predual of normal functionals on $A$ by $A_*$. We also denote by $L^2(A)$ its associated standard Hilbert space with modular conjugation $J_A$ and its $A$-bimodule structure provided by the standard normal $*$-representation $\pi_A$ and the standard normal $*$-anti-representation $\rho_A$, related by 
\[
\rho_A(a) = J_A\pi_A(a)^*J_A.
\]
If $\varphi$ is an nsf (normal, semi-finite, faithful) weight on $A$, we denote the associated GNS-map by 
\[
\Lambda_{\varphi}: \msN_{\varphi}\rightarrow L^2(A),\qquad \msN_{\varphi} = \{a\in A\mid \varphi(a^*a)<\infty\},
\]
and the associated modular operator and modular one-parameter group of automorphisms by 
\[
\nabla_{\varphi}^{it}: L^2(A) \rightarrow L^2(A),\qquad \sigma_t^{\varphi}: A \rightarrow A,\qquad t\in \R.
\]
If $\varphi$ is a normal \emph{state}, we denote $\xi_{\varphi} = \Lambda_{\varphi}(1)$ the associated GNS-vector. 

Whenever the weight/state $\varphi$ is canonically fixed, we also index the above modular data with $A$ itself in stead of $\varphi$, or with nothing at all if also $A$ is unambiguous, so 
\[
\nabla^{it} = \nabla_{A}^{it} = \nabla_{\varphi}^{it},\qquad  \sigma_t = \sigma_t^{A} = \sigma_t^{\varphi},\qquad \xi_{A} = \xi_{\varphi}.
\]
We denote by $\overline{A} = \{\overline{a}\mid a\in A\}$ the complex conjugate von Neumann algebra, i.e.\ all structure is as for $A$ except that the complex conjugate linear structure is used. 

\subsection{Compact quantum groups} 
We refer to \cite{NT14} for basic results on compact quantum groups. 

Let $\G = (M,\Delta)$ be a compact quantum group. We write the associated CQG Hopf $*$-algebra/von Neumann algebra/reduced C$^*$-algebra/universal C$^*$-algebra/Hilbert space as 
\[
\mcM = \mcO(\G),\quad M = L^{\infty}(\G),\quad \mcM_r = C_r(\G),\quad \mcM_u = C_u(\G),\quad L^2(M) = L^2(\G),
\]
using whatever notation feels more convenient. In all instances, we write the associated coproduct as $\Delta$ and the associated invariant state as $\Phi$. We freely index the associated structures below with $\G$ whenever we need disambiguation. 

We denote by $\varepsilon$ the counit and by $S$ the antipode of $(\mcM,\Delta)$. We use in this case the (unsummed) Sweedler notation 
\[
\Delta(x) = x_{(1)}\otimes x_{(2)},\qquad x\in \mcM.
\]

We denote by $\msU$ the full linear dual of $\mcM$. We endow it with the convolution $*$-algebra structure
\[
(\omega\chi)(x) = (\omega \otimes \chi)\Delta(x),\qquad \omega^*(x) = \overline{\omega(S(x)^*)},\qquad x\in \mcM,\quad \omega,\chi \in \msU.
\]
It has canonical $*$-subalgebras
\[
\mcU \subseteq W^*(\G)\subseteq \msU,
\]
where $W^*(\G)$ is the W$^*$-algebra
\[
W^*(\G) = \{x\in \msU \mid  \|x\|_{\infty}<\infty\},\qquad \|x\|_{\infty} = \sup_{\pi \textrm{ fin.-dim.} * \textrm{-rep. of }\msU} \|\pi(x)\|,
\]
and $\mcU$ is the (not necessarily unital) $*$-algebra
\[
\mcU = \{\Phi(-x)\mid x\in \mcM\} = \{\Phi(x-)\mid x\in \mcM\}.
\]

By normality of the Haar state $\Phi$, there is an obvious inclusion $\mathcal{U}\subseteq M_*$. This identification will be used in the sequel without further mention. 

We write $p_{\varepsilon}$ for the self-adjoint projection in $\mcU$ given by $\Phi$, so 
\begin{equation}\label{EqDefTrivRep}
p_{\varepsilon} = \Phi \in \mcU,\qquad p_{\varepsilon}\omega = \omega p_{\varepsilon} = \omega(1) p_\varepsilon,\qquad \forall \omega \in \msU.  
\end{equation}

The elements of $\mcM$ are analytic for the modular one-parameter group $\sigma_t$ of $(M,\Phi)$, and $\sigma_z(\mcM) \subseteq \mcM$ for all $z\in \C$. This leads to the \emph{Woronowicz characters} 
\begin{equation}\label{DefWorChar}
\delta^z := \varepsilon \circ \sigma_{iz} \in \msU.
\end{equation}
In turn, we get the following formulas for the modular automorphism group and \emph{scaling group} of $\mcM$:
\begin{equation}\label{EqFormSigmaScaling}
\sigma_{z}(x) =  \delta^{-iz/2}(x_{(1)})x_{(2)}\delta^{-iz/2}(x_{(3)}),\qquad \tau_{z}(x) =  \delta^{-iz/2}(x_{(1)})x_{(2)}\delta^{iz/2}(x_{(3)})\qquad x\in \mcM.
\end{equation}
Then $(\tau_t)_{t\in \R}$ extends to a point-$\sigma$-strong$^*$-continuous one-parameter group of automorphisms on $M$. 

We denote by $R: M \rightarrow M$ the unitary antipode. It is an involutive $*$-preserving anti-multiplicative linear isomorphism, preserving $\mcM$ and satisfying on $\mcM$ that 
\begin{equation}\label{EqFormAntipod}
S = \tau_{-i/2}\circ R = R\circ \tau_{-i/2},\qquad S^2 = \tau_{-i}. 
\end{equation}
We put 
\[
\varepsilon(\omega) = \omega(1),\qquad S(\omega) = \omega \circ S,\qquad R(\omega) = \omega\circ R,\qquad \omega \in \msU.
\]

Let $(\Rep_u(\G),\Hom_{\G}(-,-),\otimes)$ be the tensor W$^*$-category of unitary $\G$-representations. It can be viewed as either of the following (with morphisms bounded maps commuting with the respective structure): 
\begin{itemize}
\item The W$^*$-category of normal unital $*$-representations of $W^*(\G)$.
\item The W$^*$-category of non-degenerate $*$-representations $(\Hsp,\pi)$ of $\mcU$. 
\item The W$^*$-category of unitary corepresentations $U \in B(\Hsp)\bar{\otimes}M$ of $(M, \Delta)$, where 
\[
(\id\otimes \Delta)U = U_{12}U_{13}. 
\]
\end{itemize}
The first correspondence is by restriction, and the last correspondence is determined by 
\[
(\id\otimes \omega)U_{\pi} = \pi(\omega),\qquad \omega \in \mcU. 
\]
If $\mathcal{H}_\pi\in \Rep_u(\G)$ is a finite-dimensional unitary representation, we write $n_\pi:= \dim(\mathcal{H}_\pi)$ and we fix an orthonormal basis $\{e_i^\pi\}_{i=1}^{n_\pi}$ for $\mathcal{H}_\pi$.

The tensor product on $\Rep_u(\G)$ is determined by
\[
U_{\pi\otimes \pi'} = U_{\pi,13}U_{\pi',23}.
\]
If $\pi \in \Rep_{u}(\G)$ and $\xi,\eta\in \Hsp_{\pi}$, we write 
\[
U_{\pi}(\xi,\eta) = (\omega_{\xi,\eta}\otimes \id)U_{\pi}\in M
\]
for the associated \emph{matrix coefficient}. We have
\begin{align}
    \Delta(U_\pi(\xi, \eta)) = \sum_{j=1}^{n_\pi}U_\pi(\xi, e_j)\otimes U_\pi(e_j, \eta), \quad \varepsilon(U_\pi(\xi, \eta)) = \langle \xi, \eta\rangle, \quad S(U_\pi(\xi, \eta))= U_\pi(\eta, \xi)^*, \quad \xi, \eta \in \mathcal{H}_\pi.
\end{align}

We write the fixed points of a $\G$-representation as 
\[
\Hsp_{\pi}^{\G} = p_{\varepsilon}\Hsp_{\pi} = \{\xi\in \Hsp_{\pi} \mid \forall x\in \mcU: \pi(x) \xi= \varepsilon(x)\xi\},\qquad \pi \in \Rep_u(\G).
\]

We fix throughout the entire paper a choice of a maximal collection $\Irr(\G)$ of pairwise non-equivalent irreducible unitary $\G$-representations. In fact, there is an identification of $*$-algebras
\begin{equation}\label{identification}
    \mathscr{U}\cong \prod_{\pi\in \Irr(\G)} B(\mathcal{H}_\pi):  \omega \mapsto ((\id \odot \omega)(U_\pi))_{\pi\in \Irr(\G)}.
\end{equation}
We will frequently use the identification \eqref{identification} to view elements of $\mathscr{U}$ as acting on the Hilbert spaces $\mathcal{H}_\pi$ (where $\pi\in \Irr(\G)$). Concretely, an element $\omega \in \mathscr{U}$ can be represented as an element of $B(\mathcal{H}_\pi)$, uniquely determined by
\begin{equation}
    \langle \xi, \omega \eta\rangle = \omega(U_\pi(\xi, \eta)),  \quad \xi, \eta \in \mathcal{H}_\pi.
\end{equation}

If $(\Hsp_{\pi},\pi) \in \Rep_u(\G)$, its dual is the conjugate linear Hilbert space $\overline{\Hsp}_{\pi}$, endowed with the $*$-representation 
\[
\overline{\pi}(x)\overline{\xi} = \overline{\pi(R(x)^*) \xi},\qquad x\in \mcU, \quad \xi \in \mathcal{H}_\pi.
\]
With this definition for duals, the W$^*$-category $\Rep_u^{\fd}(\G)$ of finite-dimensional unitary $\G$-representations is rigid, with standard coevaluation maps in $\Rep_u^{\fd}(\G)$ determined by the vectors
\begin{equation}
t_{\pi}\in (\Hsp_{\pi}\otimes \Hsp_{\overline{\pi}})^{\G}=\Mor(\varepsilon,\pi\otimes \overline{\pi}),\qquad \langle t_{\pi},\xi \otimes \overline{\eta}\rangle = \langle \eta,\delta^{-1/4}\xi\rangle.
\end{equation}
We write 
\[
\dim_q(\pi) = \|t_{\pi}\|^2 = \|t_{\overline{\pi}}\|^2.
\]
Moreover, \begin{equation}
    U_\pi(\xi, \eta)^* = U_{\overline{\pi}}(\overline{\delta^{1/4}\xi}, \overline{\delta^{-1/4}\eta})= U_{\overline{\pi}}(\delta^{-1/4} \overline{\xi}, \delta^{1/4} \overline{\eta}), \quad \xi, \eta \in \mathcal{H}_\pi,
\end{equation}
and also
\begin{equation}
R(U_\pi(\xi, \eta))= U_{\overline{\pi}}(\overline{\eta}, \overline{\xi})= U_\pi(\delta^{-1/4}\eta, \delta^{1/4}\xi)^*, \quad \xi, \eta \in \mathcal{H}_\pi.
\end{equation}
We also have 
\begin{equation}\label{EqPresentSigma}
\sigma_z(U_{\pi}(\xi,\eta)) = U_{\pi}(\delta^{i\overline{z}/2}\xi,\delta^{-iz/2}\eta),\quad \tau_z(U_{\pi}(\xi,\eta)) = U_{\pi}(\delta^{i\overline{z}/2}\xi,\delta^{iz/2}\eta),\qquad \xi,\eta\in \Hsp_{\pi}.
\end{equation}

The \emph{Peter-Weyl orthogonality relations} for irreducible $\G$-representations $\pi,\pi'\in \Irr(\G)$ are given as follows:  
\begin{equation}\label{EqPW1}
\Phi(U_{\pi}(\xi_1,\eta_1)U_{\pi'}(\xi_2,\eta_2)^*) = \delta_{\pi,\pi'}\frac{\langle \xi_1,\xi_2\rangle \langle \eta_2,\delta^{-1/2}\eta_1\rangle}{\dim_q(\pi)},\qquad \xi_i,\eta_i\in \Hsp_{\pi},
\end{equation}
\begin{equation}\label{EqPW2}
\Phi(U_{\pi}(\xi_2,\eta_2)^*U_{\pi'}(\xi_1,\eta_1)) =\delta_{\pi,\pi'} \frac{\langle \xi_1,\delta^{1/2}\xi_2\rangle \langle \eta_2,\eta_1\rangle}{\dim_q(\pi)},\qquad \xi_i,\eta_i\in \Hsp_{\pi}.
\end{equation}

We write $W,V$ for respectively the left and right multiplicative unitary,
\begin{equation}\label{EqRightRegRep}
W^*(\Lambda(x)\otimes \Lambda(y)) = (\Lambda \otimes \Lambda)(\Delta(y)(x\otimes 1)),\quad V(\Lambda(x)\otimes \Lambda(y)) = (\Lambda \otimes \Lambda)(\Delta(x)(1\otimes y)),\qquad x,y\in M.
\end{equation}
Then $V$ is a unitary $\G$-representation in the sense above, and we obtain a faithful $*$-representation
\[
\check{\pi}(\omega) = (\id\otimes \omega)V,\quad \check{\pi}(\omega)\Lambda(x) = \Lambda(\omega(x_{(2)})x_{(1)}),\qquad \omega \in \mcU,x\in \mcM. 
\]
On the other hand, we also get a faithful $*$-representation
\[
\hat{\pi}(\omega) = (\omega\otimes \id)W,\quad \hat{\pi}(\omega)\Lambda(x) = \Lambda(\omega(S^{-1}(x_{(1)}))x_{(2)}),\qquad \omega\in \mcU,x\in \mcM. 
\]

The dual discrete quantum group $\mathbbl{\Gamma}=\check{\G}$ is given by the von Neumann algebra
$$\ell^\infty(\check{\G}):=[(\id \otimes \omega)(V): \omega \in L^\infty(\G)_*]^{\sigma\text{-weak}}$$
and the coproduct
$$\check{\Delta}: \ell^\infty(\check{\G})\to \ell^\infty(\check{\G})\bar{\otimes}\ell^\infty(\check{\G}): \check{x}\mapsto V^*(1\otimes \check{x})V.$$

It is a concretely implemented version of $W^*(\G)$, through the $*$-isomorphism
$W^*(\G)\cong \ell^\infty(\check{\G})$ uniquely determined by $\omega \cong \check{\pi}(\omega)$ for $\omega \in \mathcal{U}$. 
\section{Equivariant weak containment of correspondences via the cotensor product}\label{SecEwcc}

Throughout this section, let $\G = (M,\Delta)$ be a compact quantum group. Let $A$ be a von Neumann algebra. We mean by a \emph{right $\G$-action} on $A$ an isometric unital, normal $*$-homomorphism 
$\alpha: A\rightarrow A\bar{\otimes} M$ satisfying the coaction property $(\id \otimes \Delta)\alpha = (\alpha\otimes \id)\alpha$. We denote such a $\G$-action by $\alpha: A\curvearrowleft \G$. Similarly, we define the notion of a \emph{left $\G$-action} $\alpha: \G \curvearrowright A$.

One can switch between right and left $\G$-actions by 
\[
(A,\alpha) \leftrightarrow (\overline{A},\alpha_c),\qquad \alpha_c = (R(-)^* \bar{\otimes} \conj_A)\circ \flip\circ \alpha\circ\conj_{\overline{A}}
\]
and vice versa, where $\conj_A(a)=\overline{a}, \conj_{\overline{A}}(\overline{a}) = a$ and $\flip(a\otimes x) = x\otimes a$. 

We mean by right/left $\G$-W$^*$-algebra a von Neumann algebra $A$ equipped with a right/left $\G$-action. Both right and left $\G$-actions appear in this paper, although we will favour right actions. 

If $(A,\alpha)$ is a right/left $\G$-W$^*$-algebra, we denote the \emph{fixed point algebra} by
\[
\textrm{(right)}\quad A^\G = \{a\in A: \alpha(a)= a\otimes 1\},\qquad \textrm{(left)}\quad A^\G = \{a\in A: \alpha(a)= 1\otimes a\}.\] 

\subsection{Equivariant correspondences}

We briefly recall some of the theory developed in \cite{DCDR23}.

\begin{Def}\label{DefEqCorr}
Let $(A,\alpha),(B,\beta)$ be (right) $\G$-W$^*$-algebras. A \emph{$\G$-$A$-$B$-correspondence} $\Hsp = (\Hsp,\pi,\rho,U)$ consists of a Hilbert space $\Hsp$ with 
\begin{itemize}
\item a normal unital $*$-representation $\pi$ of $A$ on $\Hsp$, 
\item a normal unital anti-$*$-representation $\rho$ of $B$ on $\Hsp$, and 
\item a unitary representation $U \in B(\Hsp)\bar{\otimes}M$ of $\G$ on $\Hsp$,
\end{itemize} 
satisfying for all $a\in A$ and $b\in B$ the compatibility relations
\begin{equation}\label{EqDefCondGeqCorr}
\pi(a)\rho(b) = \rho(b)\pi(a),\qquad U(\pi(a)\otimes 1)U^* = (\pi\otimes \id)\alpha(a),\qquad  U^*(\rho(b)\otimes 1)U = (\rho \otimes R)(\beta(b)).
\end{equation}
We denote by $\Corr^{\G}(A,B)$ the W$^*$-category of $\G$-equivariant $A$-$B$-correspondences.
\end{Def}

As shown in \cite{DCDR23}, the theory of general $\G$-equivariant correspondences is governed by a particular C$^*$-algebra $C^{\G}(A,B)$, constructed as a separation-completion of 
\[
A_s \odot M_*\odot B_s, \qquad A_s = [(\id\otimes \omega)\alpha(a)\mid a\in A,\omega\in M_*]
\] 
under the seminorm 
\[
\|\sum_{i=1}^n a_i \otimes \omega_i \otimes b_i\|= \sup\{\|\sum_{i=1}^n \pi(a_i)U(\omega_i)\rho(b_i)\|\},\qquad n \in \N,a_i\in A_s,b_i \in B_s,\omega_i \in M_*,
\]
the supremum running over all $\G$-equivariant $A$-$B$-correspondences. The C$^*$-algebra structure of $C^{\G}(A,B)$ is determined by the fact that every $\G$-$A$-$B$-correspondence $\mathcal{H}$ induces a non-degenerate $*$-homomorphism
\begin{equation}\label{EqDefLowerTheta}
\theta_{\mathcal{H}}: C^\G(A,B)\to B(\mathcal{H}): a \otimes \omega \otimes b \mapsto \pi(a)U(\omega)\rho(b).
\end{equation}

Given $\mathcal{H}, \mathcal{G}\in \Corr^\G(A,B)$, we say that $\mathcal{H}$ is equivariantly weakly contained in $\mathcal{G}$, and we write $\mathcal{H}\preccurlyeq\mathcal{G}$, if 
$$\|\theta_{\mathcal{H}}(x)\|\le \|\theta_{\mathcal{G}}(x)\|, \quad \forall x\in C^\G(A,B).$$

In general, it is a hard problem to decide if a given equivariant correspondence is weakly contained in another equivariant correspondence. We will be especially interested in determining if a given correspondence $\mathcal{H}\in \Corr^\G(A,A)$ weakly contains the trivial $\G$-$A$-$A$-correspondence $L^2(A)$. The main goal of this section is to provide a more algebraic strategy to do this in the context of compact quantum groups. This strategy lies at the basis of this entire paper.

As $\G$ is compact, the $C^*$-algebra $C^{\G}(A,B)$ is easier to understand than in the general locally compact case treated in \cite{DCDR23}. Recall first that if $(A,\alpha)$ is a right/left $\G$-W$^*$-algebra, its \emph{algebraic core}
\[
\textrm{(right)}\quad \mcA = \operatorname{span}\{(\id\otimes \omega)(\alpha(a))\mid a\in A,\omega \in \mcU\},\qquad \textrm{(left)}\quad \mcA = \operatorname{span}\{(\omega\otimes \id)(\alpha(a))\mid a\in A,\omega \in \mcU\}
\] 
is a $\sigma$-weakly dense unital $*$-subalgebra of $A$, to which $\alpha$ restricts as an ordinary, right/left $\mcM$-comodule $*$-algebra structure. We will use for this comodule the standard Sweedler notation,  
\[
\textrm{(right)}\quad \alpha: \mcA \rightarrow \mcA\odot \mcM,\;  \alpha(a) = a_{(0)}\otimes a_{(1)},\qquad \textrm{(left)}\quad
\alpha: \mcA \rightarrow \mcM\odot \mcA,\; \alpha(a) = a_{(-1)}\otimes a_{(0)}.
\]
In particular, the index $(0)$ always corresponds to $\mathcal{A}$, and non-zero indices always correspond to $\mcM$.

Note that the above definition of algebraic core indeed coincides with the one provided by \eqref{EqAlgCoreOth}.

We note that $\mcA$ admits a universal C$^*$-envelope, and that moreover any $*$-representation of $\mcA$ on a pre-Hilbert space is automatically bounded.

\begin{Def}[\cite{AS21}]
Let $(A,\alpha)$ be a right $\G$-W$^*$-algebra, and $(B,\beta)$ a left $\G$-W$^*$-algebra, with respective algebraic cores $\mcA,\mcB$.

The \emph{extended double crossed product} $*$-algebra $\mcA\rtimes \msU\ltimes \mcB$ is the vector space $\mcA \odot \msU \odot \mcB$ with the $*$-algebra structure  
\begin{align*}
&(a\otimes \omega \otimes b) (a'\otimes \omega'\otimes b') = aa_{(0)}'\otimes \omega(a_{(1)}'-)\omega'(-b_{(-1)})\otimes b_{(0)}b',\\
&(a\otimes \omega\otimes b)^* = a_{(0)}^*\otimes \omega^*(a_{(1)}^*-b_{(-1)}^*)\otimes b_{(0)}^*.
\end{align*}
The \emph{double crossed product} $*$-algebra $\mcA\rtimes \G\ltimes \mcB$ is the $*$-subalgebra $\mcA\odot \mcU \odot \mcB$. 
\end{Def}
It is elementary to verify that $\mcA\rtimes \msU\ltimes \mcB$ can also be viewed as the universal unital $*$-algebra generated by copies of $\mcA,\mcB$ and $\msU$ satisfying the above commutation relations. So, we simply write expressions in this algebra as
\[
a\omega = a\otimes \omega\otimes 1,\qquad \omega a b= a_{(0)}\otimes \omega(a_{(1)}-)\otimes b,\qquad \textrm{etc.}
\]

 If $\mathcal{A}= \mathbb{C}$, we recover the $*$-algebra $\G \ltimes \mathcal{B}$. If $\mathcal{B}= \mathbb{C}$, we recover the $*$-algebra $\mathcal{A}\rtimes \G$.
 
Also, any product involving elements of $\mcA,\mcB$ and at least one element of $\mcU$ will land in $\mcA\rtimes \G\ltimes \mcB$. 
It can be easily shown that $\mcA\rtimes \G\ltimes \mcB$ allows a universal C$^*$-envelope $C_u^*(\mcA\rtimes \G\ltimes \mcB)$.

It will be convenient to introduce some terminology for a more algebraic version of the notion of $\G$-$A$-$B$-correspondence:

\begin{Def} Let $(A, \alpha)$ and $(B, \beta)$ be two right $\G$-$W^*$-algebras. An algebraic $\G$-$\mathcal{A}$-$\mathcal{B}$-correspondence consists of the data $(\mathcal{H}, \pi, \rho, U)$ such that \begin{itemize}
\item $\mathcal{H}$ is a Hilbert space.
        \item $\pi: \mathcal{A}\to B(\mathcal{H})$ is a unital $*$-representation.
        \item $\rho: \mathcal{B}\to B(\mathcal{H})$ is a unital anti $*$-representation.
        \item The images of $\pi$ and $\rho$ commute.
        \item $U \in B(\mathcal{H})\bar{\otimes}L^\infty(\G)$ is a  unitary $\G$-representation.
        \item $(\pi\odot \id)\alpha(a) = U(\pi(a)\otimes 1)U^*$ for all $a\in \mathcal{A}$.
        \item $(\rho\odot R)\beta(b) = U^*(\rho(b)\otimes 1)U$ for all $b\in \mathcal{B}.$
    \end{itemize}
\end{Def}

Note that every $\G$-$A$-$B$-correspondence gives rise to an algebraic $\G$-$\mathcal{A}$-$\mathcal{B}$-correspondence.

\begin{Prop}\label{algebraic} Given an algebraic $\G$-$\mathcal{A}$-$\mathcal{B}$-correspondence $\mathcal{H}=(\mathcal{H}, \pi, \rho, U)$, the map
$$\theta^{\mathcal{H}}: \mathcal{A}\rtimes \G \ltimes \overline{\mathcal{B}}\to B(\mathcal{H}): a \omega \overline{b}\mapsto \pi(a) U(\omega) \rho(b^*)$$
is a non-degenerate $*$-representation. The assignment $\mathcal{H}\mapsto \theta^{\mathcal{H}}$ sets up a bijective correspondence between the algebraic $\G$-$\mathcal{A}$-$\mathcal{B}$-correspondences and the non-degenerate $*$-representations of $\mathcal{A}\rtimes \G \ltimes \overline{\mathcal{B}}$ on Hilbert spaces.
\end{Prop}
\begin{proof}
Given an algebraic $\G$-$\mathcal{A}$-$\mathcal{B}$-correspondence $\mathcal{H}$, routine calculations show that $\theta^{\mathcal{H}}$ is a $*$-representation. It is non-degenerate since $\pi, \rho$ are unital and since the $*$-representation $\mathcal{U}\to B(\mathcal{H}): \omega \mapsto U(\omega)$ is non-degenerate.

Conversely, assume that $\theta: \mathcal{A}\rtimes \G \ltimes \overline{\mathcal{B}}\to B(\mathcal{H})$ is a non-degenerate $*$-representation. By a standard argument, there is a unique $*$-representation $\pi_\theta: \mathcal{A}\to B(\mathcal{H})$ such that
$$\pi_\theta(a)\theta(z)\xi = \theta (az)\xi, \quad a \in \mathcal{A}, \quad z \in \mathcal{A}\rtimes \G \ltimes \overline{\mathcal{B}}.$$
Similarly, there is a unique anti-$*$-representation $\rho_\theta: \mathcal{B}\to B(\mathcal{H})$ such that
$$\rho_\theta(b)\theta(z)\xi =\theta(\overline{b^*}z)\xi, \quad b \in B, \quad z \in \mathcal{A}\rtimes \G \ltimes \overline{\mathcal{B}}.$$
There is a unique unitary $U_\theta \in B(\mathcal{H})\bar{\otimes} L^\infty(\G)$ such that
$$\theta(\omega) = U_\theta(\omega), \quad \omega \in \mathcal{U}.$$
It is then easily verified that $\mathcal{H}=(\mathcal{H,}\pi_\theta, \rho_\theta, U_\theta)$ is an algebraic $\G$-$\mathcal{A}$-$\mathcal{B}$-correspondence such that $\theta = \theta^{\mathcal{H}}$. 
\end{proof}

\begin{Lem}\label{LemIsoCstar}
Assume $A,B$ are right $\G$-W$^*$-algebras. Then the map 
\begin{equation}\label{StarMorph}
    \theta_u:\mcA \rtimes \G\ltimes \overline{\mcB}\rightarrow C^{\G}(A,B),\qquad a\omega \overline{b}\mapsto a \otimes \omega \otimes b^*
\end{equation}
is a well-defined $*$-homomorphism with dense range.
\end{Lem} 
\begin{proof} Recall from \cite{DCDR23} that there is a $C^*$-algebra isomorphism
$$C^\G(A,B) \cong C^\G_{\mathcal{H}_u}(A,B)= [\pi_u(a)U_u(\omega)\rho_u(b): a \in A_s, \omega \in M_*, b \in B_s]: a\otimes \omega \otimes b\mapsto \pi_u(a)U_u(\omega)\rho_u(b),$$
with $\mathcal{H}_u$ a universal $\G$-$A$-$B$ correspondence. By Proposition \ref{algebraic}, the map
$$\mathcal{A}\rtimes \G \ltimes \overline{\mathcal{B}}\to C^\G_{\mathcal{H}_u}(A,B): a \omega \overline{b}\mapsto \pi_u(a) U_u(\omega) \rho_u(b^*)$$
is a $*$-homomorphism. Therefore, the map $\theta_u$ is a $*$-homomorphism as well. Since $\mcA \subseteq A_s$ is norm-dense and $\mcU \subseteq M_*$ is norm-dense, the range of $\theta_u$ must be dense.
\end{proof}

We end this section with an observation on standard representations. If $(A,\alpha)$ is a (right) $\G$-von Neumann algebra, we can always choose a $\G$-invariant nsf weight $\phi$ on $A$, e.g.\ take an nsf weight $\phi_0$ on $A^{\G}$ and put
\[
\phi = \phi_0 \circ E_{\G},
\]
where $E_{\G}$ is the faithful normal conditional expectation 
\[
E_{\G} = (\id\otimes \Phi)\alpha: A \rightarrow A^{\G}.
\]
The unitary $\G$-representation on $L^2(A)$ that is the canonical unitary implementation of $\alpha$ \cite{Vae01} is then simply given by 
\begin{equation}\label{action of dual}
U_{\alpha}(\omega)\Lambda_{\phi}(a) = \Lambda_{\phi}((\id\otimes \omega)\alpha(a)),\qquad \omega \in \mcU,a\in \msN_{\phi}, 
\end{equation}
and the action of $\mcA\rtimes \G\ltimes \overline{\mcA}$ on $L^2(A)$ is determined by 
\[
\theta^{L^2(A)}(a\omega \overline{b}) = \pi_A(a)U_\alpha(\omega)\rho_A(b^*),\qquad a,b\in \mcA,\omega\in \mcU. 
\]

Assume now that $A, B$ are (right) $\G$-W$^*$-algebras, and assume moreover that we have a (not necessarily unital) $\G$-equivariant normal embedding of von Neumann algebras
\[
A \subseteq B. 
\]
Then we have in particular that $\mcA \subseteq \mcB$ as a (not necessarily unital) $*$-subalgebra, and we can also consider 
\[
\mcA \rtimes \G \ltimes \overline{\mcA} \subseteq \mcB\rtimes \G\ltimes \overline{\mcB}
\]
as a $*$-subalgebra.

The following result will be useful at some point (cf.\ \cite{AD93}*{Prop. 2.1} for the non-equivariant setting). 
\begin{Lem}\label{LemEstimateNorm}
Assume $A,B$ as above, and assume that there exists a $\G$-equivariant normal conditional expectation $E: B \rightarrow A$. Then 
\[
\|\theta^{L^2(A)}(z)\|\leq \|\theta^{L^2(B)}(z)\|,\qquad z\in \mcA \ltimes \G \rtimes \overline{\mcA}.
\]
\end{Lem} 
\begin{proof}
The lemma is trivial if $B = A\oplus C$ as a direct sum of $\G$-W$^*$-algebras.

Assume now that $A\subseteq B$ is a unital inclusion and $E$ is faithful. Then we may pick $\phi_A$ a $\G$-invariant nsf weight on $A$, and take $\phi_B := \phi_A\circ E$ as $\G$-invariant nsf weight on $B$. Then $(\sigma_t^{\phi_B})_{\mid A} = \sigma_t^{\phi_A}$, and so we obtain an isometric inclusion of $\G$-equivariant $A$-$A$-correspondences 
\[
L^2(A) \rightarrow L^2(B),\quad \Lambda_{\phi_A}(a) \mapsto \Lambda_{\phi_B}(a),\qquad a\in \msN_{\phi_A}. 
\]

If $E$ is not faithful (but the inclusion is unital), let $p \in B$ be its support projection,
\[
Bp^{\perp} = \{b\in B\mid E(b^*b) =0\}. 
\]
Since $Bp^{\perp}$ is stable under the $\G$-action, we will have $p\in B^{\G}$. Moreover, $p$ commutes with $A$, and $A\cong pA$ $\G$-equivariantly through $a \cong pa$. We can then conclude by applying the previous steps to 
\[
A\cong pA \subseteq pBp \subseteq pBp + p^{\perp}Bp^{\perp}\subseteq B.
\]
Finally, assume that $A \subseteq B$ is not necessarily unital. Then we can apply the previous steps to 
\[
A \subseteq 1_AB1_A \subseteq 1_AB1_A+1_A^{\perp}B1_A^{\perp}\subseteq B.
\]
\end{proof}

\subsection{Cotensor products}

\begin{Def}
For $(A,\alpha)$ a right $\G$-W$^*$-algebra and $(B,\beta)$ a left $\G$-W$^*$-algebra, we define the unital $*$-algebra
\[
 \mcA\overset{\G}{\square} \mcB := \{z\in \mcA\odot \mcB \mid (\alpha\odot \id)z = (\id\odot \beta)z\},
\]
and we refer to it as the cotensor product of $\mathcal{A}$ and $\mathcal{B}$. When $\G$ is clear from context, we will sometimes leave it out, and simply write $\mcA \square \mcB$.
\end{Def}

In the next result, we identify the cotensor product as a corner of the double crossed product $*$-algebra.

\begin{Prop} Let $(A, \alpha)$ a right $\G$-$W^*$-algebra and let $(B, \beta)$ be a left $\G$-$W^*$-algebra. 
The following map is an isomorphism of unital $*$-algebras:  
\begin{equation}\label{EqIsoCotens}
\kappa:\mcA \overset{\G}{\square}\mcB \rightarrow p_{\varepsilon}(\mcA\rtimes \G\ltimes \mcB)p_{\varepsilon},\qquad \sum_{j=1}^n a_j \otimes b_j \mapsto \sum_{j=1}^n a_j \delta^{1/2} b_j p_{\varepsilon}.  
\end{equation}
The inverse map is given by the restriction of the map 
\begin{equation}\label{EqMapE}
E: \mcA\rtimes \G\ltimes \mcB \rightarrow \mcA\overset{\G}{\square}\mcB,\qquad ab\omega \mapsto \omega(1) \Phi(a_{(1)} \sigma_{-i/2}R(b_{(-1)}))a_{(0)}\otimes b_{(0)}. 
\end{equation}
\end{Prop} 
\begin{proof}
Clearly $\kappa$ is well-defined as a linear map
\[
\mcA \overset{\G}{\square}\mcB \rightarrow (\mcA\rtimes \msU\ltimes \mcB)p_{\varepsilon}
 = (\mcA\rtimes \G\ltimes \mcB)p_{\varepsilon}.
\]
Using the commutation relations of $\mcA\rtimes \msU\ltimes \mcB$ to rewrite
\[
\kappa(\sum_j a_j \otimes b_j) = \sum_j a_j \delta^{-1/2}(b_{j(-1)})b_{j(0)}p_{\varepsilon},  
\]
it is also clear that $\kappa$ is injective.

Now if $a\in \mcA,b\in \mcB$, an easy computation shows that 
\begin{equation}\label{EqCutDown}
p_{\varepsilon}abp_{\varepsilon} = \Phi(a_{(1)}S^{-1}(b_{(-1)}))a_{(0)}b_{(0)}p_{\varepsilon}.
\end{equation}
Write $z =  \Phi(a_{(1)}S^{-1}(b_{(-1)}))a_{(0)}\otimes b_{(0)}$. By means of the identity 
\[
\Phi(x_{(2)}y)x_{(1)} = \Phi(xy_{(2)})S^{-1}(y_{(1)}),\qquad x,y\in \mcM,
\]
it is easily checked that 
\[
(\alpha\otimes \id)z = (\id\otimes (S^{-2}\otimes \id)\beta)z. 
\]
Hence, by \eqref{EqFormSigmaScaling} and \eqref{EqFormAntipod} we get 
\[
\mcA\overset{\G}{\square}\mcB\ni \widetilde{z} := \Phi(a_{(1)}S^{-1}(b_{(-2)}))a_{(0)}\otimes \delta^{1/2}(b_{(-1)})b_{(0)}  \\ = \Phi(a_{(1)} \sigma_{-i/2}R(b_{(-1)}))a_{(0)}\otimes b_{(0)} = E(abp_{\varepsilon}), 
\]
and $\kappa(\widetilde{z}) = p_{\varepsilon}abp_{\varepsilon}$. This shows that $E$ has the correct range, and moreover provides a right inverse for $\kappa$. In fact, we see $\kappa(E(x)) = p_{\varepsilon}xp_{\varepsilon}$ for general $x\in  \mcA\rtimes \G\ltimes \mcB$.

On the other hand, the same computation as above shows that 
\[
\kappa(z) = p_{\varepsilon}\kappa(z),\qquad z\in \mcA\overset{\G}{\square}\mcB,
\]
so indeed \eqref{EqIsoCotens} is a linear isomorphism with the restriction of $E$ as its inverse. 

Now consider the $*$-representation of $\mcA\rtimes \G\ltimes \mcB$ on $L^2(A)\otimes L^2(\G)\otimes L^2(B)$ determined by 
\begin{multline*}
\theta_{\coarse}(a)= (\pi_A\otimes \pi_M)\alpha(a)_{12},\qquad \theta_{\coarse}(b) = ((\rho_M\circ R)\otimes \pi_B)\beta(b)_{23},\\
 \theta_{\coarse}(\omega) = 1\otimes \check{\pi}(\omega) \otimes 1,\qquad a\in \mcA,b\in\mcB,\omega \in \mcU. 
\end{multline*}
Then identifying $L^2(A)\otimes L^2(B) \cong L^2(A)\otimes \C\xi_{\Phi} \otimes L^2(B)$, it is easily seen that 
\[
\theta_{\coarse}(p_{\varepsilon}xp_{\varepsilon}) = (\pi_A\otimes \pi_B)(E(x)),\qquad x \in \mcA\rtimes \G\ltimes \mcB,
\]
showing that $E$ is a unital $*$-homomorphism on $p_{\varepsilon}(\mcA\rtimes \G\ltimes \mcB) p_{\varepsilon}$.
\end{proof} 

\begin{Def} Given a right $\G$-$W^*$-algebra $(A, \alpha)$ and a left $\G$-$W^*$-algebra $(B, \beta)$, 
we define the \emph{universal C$^*$-norm} on $\mcA \square \mcB$ by 
\[
\|z\|_u = \|\kappa(z)\|,\qquad z \in \mcA \square \mcB,
\]
where on the right we borrow the C$^*$-norm of $C_u^*(\mathcal{A}\rtimes \G \ltimes \mathcal{B})$. We denote the norm-completion of $\mcA \square \mcB$ for the above norm by $C^*_u(\mcA\square \mcB)$. 
\end{Def}

\begin{Rem} 
We stress that $C_u^*(\mathcal{A}\square \mathcal{B})$ is in general \emph{not} the universal C$^*$-envelope of $\mathcal{A}\square \mathcal{B}$. Such a universal $C^*$-envelope might not even exist. For example, in Proposition \ref{Prop}, we will see an example where the cotensor product is identified with the $*$-algebra $\operatorname{Fus}[\mathbb{H}]$ associated to a compact quantum group $\mathbb{H}$. The $*$-algebra $\operatorname{Fus}[SU_q(2)]$ is $*$-isomorphic to the polynomial $*$-algebra $\mathbb{C}[X]$ in the indeterminate $X$, which is a $*$-algebra that does not admit a universal $C^*$-envelope.
\end{Rem} 
Clearly, the map $\kappa$ of \eqref{EqIsoCotens} extends uniquely to a unital $*$-isomorphism \begin{equation}\label{EqIsoCotens2}
   \kappa: C^*_u(\mcA\square \mcB) \to p_{\varepsilon}C_u^*(\mathcal{A}\rtimes \G \ltimes \mathcal{B})p_{\varepsilon}. 
\end{equation}

Assume now again that $(A, \alpha)$ and $(B, \beta)$ are both two \emph{right} $\G$-$W^*$-algebras. We recall that $\overline{B}$ then becomes a left $\G$-$W^*$-algebra in a canonical way.
If $\Hsp \in \Corr^{\G}(A,B)$, we denote by $\theta_{\square}^{\Hsp}$ the corresponding $*$-representation of $\mcA\square \overline{\mcB}$ or of $C^*_u(\mcA\square \overline{\mcB})$ on $\theta^{\Hsp}(p_{\varepsilon})\Hsp$, so 
\begin{equation}\label{EqRestrRep}
\theta_{\square}^{\Hsp}(z)\xi =  \sum_{j=1}^n \pi(a_j) \delta^{1/2}(b_{j(1)}^*) \rho(b_{j(0)}^*)\xi,\qquad \xi\in \theta^{\Hsp}(p_{\varepsilon})\Hsp, \quad z = \sum_{j=1}^n a_j \otimes \overline{b_j} \in \mcA \square \overline{\mcB}.
\end{equation}

Before giving the proof of the next important result, let us recall some basic $C^*$-algebra theory in three bullet points (see e.g.\ \cite{RW98}*{Chapters 2,3} for a modern reference). If $C$ is a $C^*$-algebra and $(\mathcal{H}, \pi), (\mathcal{H}', \pi')$ $*$-representations of $C$ on Hilbert spaces, we write $\pi \preccurlyeq \pi'$, and say that $\pi$ is weakly contained in $\pi'$, if $\operatorname{Ker}(\pi')\subseteq \operatorname{Ker}(\pi)$, or equivalently $\|\pi(c)\|\le \|\pi'(c)\|$ for all $c\in C$. 

\begin{itemize}
\item If $C,D$ are $C^*$-algebras, ${}_C\mathscr{E}_D$ is an imprimitivity bimodule and $(\mathcal{H}, \pi)$ is a $*$-representation of $D$, we can form the Hilbert space 
$\mathscr{E}\otimes_D \mathcal{H}$ to be the separation-completion of $\mathscr{E}\odot_D \mathcal{H}$ with respect to the semi-inner product uniquely determined by
$$\langle \xi \otimes_D \eta, \xi'\otimes_D \eta'\rangle=\langle\eta,  \pi(\langle \xi, \xi'\rangle_D)\eta'\rangle.$$
There is a unique $*$-representation
$$\operatorname{Ind}(\pi, \mathscr{E}): C \to B(\mathscr{E}\otimes_D \mathcal{H}), \quad \operatorname{Ind}(\pi, \mathscr{E})(c)(\xi\otimes_D \eta) = c\xi \otimes_D \eta.$$
Given two $*$-representations $(\mathcal{H}, \pi), (\mathcal{H}', \pi')$ of $D$, we have 
$$\pi \preccurlyeq \pi' \iff \operatorname{Ind}(\pi, \mathscr{E}) \preccurlyeq \operatorname{Ind}(\pi', \mathscr{E}).$$
    \item If $C$ is a $C^*$-algebra and $p\in C$ is a projection, consider the ideal $I:= [CpC]$ in $C$. The imprimitivity $pCp$-$I$-bimodule $pC$ defines a strong Morita equivalence between the $C^*$-algebras $pCp$ and $I$. Given a non-degenerate $*$-representation $(\mathcal{H}, \pi)$ of $I$, there is an associated $*$-representation
    $$\tilde{\pi}: pCp\to B(\pi(p)\mathcal{H}), \quad \tilde{\pi}(pcp)\xi = \pi(pc)\xi, \quad c\in C, \quad \xi \in \pi(p)\mathcal{H}.$$
    There is a unitary intertwiner
    $$(pC\otimes_I \mathcal{H}, \operatorname{Ind}(\pi, pC)) \to (\pi(p)\mathcal{H}, \tilde{\pi}): pc \otimes_I \xi \mapsto \pi(pc)\xi$$
    of $pCp$-representations. Using the first bullet point, we conclude that 
    $$(\mathcal{H}, \pi)\preccurlyeq (\mathcal{H}', \pi')\iff (\pi(p)\mathcal{H}, \tilde{\pi})\preccurlyeq (\pi'(p)\mathcal{H}', \tilde{\pi}').$$
    \item Let $C$ be a $C^*$-algebra and $I$ a norm-closed two-sided ideal in $C$. Given a $*$-representation $(\mathcal{H}, \pi)$ of $C$, let us consider the restricted $*$-representation $\pi_I: I \to B([\pi(I)\mathcal{H}])$. If  $(\mathcal{H}', \pi')$ is another $*$-representation of $C$, then 
    $$\pi \preccurlyeq \pi'\implies \pi_I \preccurlyeq \pi'_I,$$
    and the converse holds if $[\pi(I)\mathcal{H}] = \mathcal{H}$.
\end{itemize}

\begin{Prop}
    \label{weak containment}
   Let $A$ be a $\G$-$W^*$-algebra and $\mathcal{H}\in \Corr^\G(A,A)$. The following statements are equivalent:
\begin{enumerate}
    \item $L^2(A)\preccurlyeq \mathcal{H}$, i.e.\  $\|\theta_{L^2(A)}(x)\|\le \|\theta_{\mathcal{H}}(x)\|$ for all $x\in C^\G(A,A).$
    \item $\theta_\square^{L^2(A)}\preccurlyeq \theta_\square^{\mathcal{H}}$, i.e.\ $\|\theta_\square^{L^2(A)}(x)\| \le \|\theta_\square^{\mathcal{H}}(x)\|$ for all $x\in \mathcal{A}\square \overline{\mathcal{A}}.$
\end{enumerate}
\end{Prop}
\begin{proof} Recall from \eqref{StarMorph} the map $\theta_u: \mathcal{A}\rtimes \G \ltimes \overline{\mathcal{A}}\to C^\G(A,A)$. Given $\mathcal{H}\in \Corr^\G(A,A)$, we have $\theta_{\mathcal{H}}\circ \theta_u = \theta^{\mathcal{H}}$. Since $\theta_u$ has dense image, it is clear that $(1)$ is equivalent to
$$\|\theta^{L^2(A)}(x)\|\le \|\theta^{\mathcal{H}}(x)\|, \quad x \in C_u^*(\mathcal{A}\rtimes \G \ltimes \overline{\mathcal{A}}).$$

Consider now the two-sided ideal 
$$I:= [C_u^*(\mathcal{A}\rtimes \G \ltimes \overline{\mathcal{A}})p_\varepsilon C_u^*(\mathcal{A}\rtimes \G \ltimes \overline{\mathcal{A}}) ]\subseteq C_u^*(\mathcal{A}\rtimes \G \ltimes \overline{\mathcal{A}}).$$
It is easy to see that 
$[\theta^{L^2(A)}(I) L^2(A)]= L^2(A)$, so that the last bullet point above shows that $(1)$ is equivalent with  
$$\|\theta_I^{L^2(A)}(x)\| \le \|\theta^{\mathcal{H}}_I(x)\|, \quad x \in I.$$
Combining the second bullet point above with the isomorphism \eqref{EqIsoCotens2}, the equivalence $(1) \iff (2)$ follows.
\end{proof}

\subsection{Ergodic actions}

We end this section by observing that in the ergodic setting, some of the results in this section can be upgraded:

\begin{Def}
We call a $\G$-dynamical von Neumann algebra $A$ \emph{ergodic} if $A^\G= \mathbb{C}1$.
\end{Def}

Recall that if $(A, \alpha)$ is a (right) $\G$-$W^*$-algebra, then the crossed product von Neumann algebra is given by
$$A\rtimes_\alpha \G := [\alpha(a)(1\otimes \check{\pi}(\omega)): a \in A, \omega \in \mathcal{U}]''\subseteq A \bar{\otimes} B(L^2(\G)).$$ We then have that in fact
$$A\rtimes_\alpha \G := [\alpha(a)(1\otimes \check{\pi}(\omega)): a \in A, \omega \in \mathcal{U}]^{\sigma\textrm{-weak}}.$$

The following result is probably well-known. For the convenience of the reader, we provide a proof.

\begin{Prop}\label{universalcrossedproduct}
     Let $(A, \alpha)$ be a right ergodic $\G$-$W^*$-algebra. We can identify
     $$A\rtimes_\alpha \G \cong C_u^*(\mathcal{A}\rtimes_\alpha \G)^{**}: \alpha(a)(1\otimes \check{\pi}(\omega)) \cong a\omega, \quad a \in \mathcal{A},\quad \omega \in \mathcal{U}.$$
     Consequently, if $\pi: \mathcal{A}\rtimes_\alpha \G \to B(\mathcal{H})$ is a non-degenerate $*$-representation, there exists a unique unital normal $*$-representation
    $\pi': A\rtimes_\alpha \G \to B(\mathcal{H})$
    such that the diagram 
    $$
\begin{tikzcd}
\mathcal{A}\rtimes_\alpha \G \arrow[rr, "\pi"] \arrow[d, hook] &  & B(\mathcal{H}) \\
A\rtimes_\alpha \G \arrow[rru, "\pi'"', dashed]  &  &               
\end{tikzcd}$$
commutes.
\end{Prop}
\begin{proof} 
As $(A,\alpha)$ is ergodic, it follows from \cite{Boc95} that $\mcA \rtimes_{\alpha}\G$ is a directed union of finite-dimensional C$^*$-subalgebras, with $$C_u^*(\mathcal{A}\rtimes_\alpha \G)\cong \bigoplus_{i\in I}^{c_0} \mathcal{K}(\mathcal{H}_i)$$
    for a collection of Hilbert spaces $\{\mathcal{H}_i\}_{i\in I}$.
In particular, $\mathcal{A}\rtimes_\alpha \G$ admits a unique $C^*$-norm, and the 
non-degenerate $*$-representation
    $$C_u^*(\mathcal{A}\rtimes_\alpha \G)\to B(L^2(A)\otimes L^2(\G)): a \omega \mapsto (\pi_A\otimes \id)(\alpha(a))(1\otimes \check{\pi}(\omega))$$ is necessarily faithful. 
    
By  the universal property of the bidual, the above $*$-homomorphism therefore extends uniquely to a unital normal $*$-representation
    $$\prod_{i\in I}^{\ell^\infty} B(\mathcal{H}_i) \cong C_u^*(\mathcal{A}\rtimes_\alpha \G)^{**}\to B(L^2(A)\otimes L^2(\G)).$$
    Since $\bigoplus_{i\in I}^{c_0} \mathcal{K}(\mathcal{H}_i)$ is an essential ideal of $\prod_{i\in I}^{\ell^\infty} B(\mathcal{H}_i)$, the $*$-morphism
    $$C_u^*(\mathcal{A}\rtimes_\alpha \G)^{**}\to B(L^2(A)\otimes L^2(\G))$$
    is faithful as well. We conclude that 
    $C_u^*(\mathcal{A}\rtimes_\alpha \G)^{**}\cong A\rtimes_\alpha \G.$
\end{proof}

The following result should be compared with Proposition \ref{algebraic}.
\begin{Prop}\label{ErgodicCorrespondence}
    Let $(A, \alpha)$ and $(B, \beta)$ be ergodic $\G$-$W^*$-algebras. The assignment $\mathcal{H}\mapsto \theta^{\mathcal{H}}$ sets up a bijective correspondence between the $\G$-$A$-$B$-correspondences and the non-degenerate $*$-representations of $\mathcal{A}\rtimes \G \ltimes \overline{\mathcal{B}}$ on Hilbert spaces. 
    
    Consequently, the $W^*$-category $\Corr^\G(A,B)$ of $\G$-$A$-$B$-correspondences and the $W^*$-category $\operatorname{Rep}_*(\mathcal{A}\rtimes \G \ltimes \overline{\mathcal{B}})$ of non-degenerate $*$-representations of $\mathcal{A}\rtimes \G \ltimes \overline{\mathcal{B}}$ on Hilbert spaces are isomorphic through an isomorphism that is the identity on morphisms.
\end{Prop}
\begin{proof}
    It suffices to show that every algebraic $\G$-$\mathcal{A}$-$\mathcal{B}$-correspondence $(\mathcal{H}, \pi, \rho, U)$ extends (necessarily in a unique way) to a $\G$-$A$-$B$-correspondence $(\mathcal{H}, \tilde{\pi}, \tilde{\rho}, U)$. We have a non-degenerate $*$-representation
    $$\pi\rtimes U: \mathcal{A}\rtimes_\alpha \G \to B(\mathcal{H}): a \omega \mapsto \pi(a)U(\omega).$$
    By Proposition \ref{universalcrossedproduct}, it extends uniquely to a normal unital $*$-representation
    $A\rtimes_\alpha \G \to B(\mathcal{H})$. Composing with the inclusion $\alpha: A \to \alpha(A)\subseteq A\rtimes_\alpha \G$, we find a normal extension $\tilde{\pi}: A \to B(\mathcal{H})$ of $\pi$. Similarly, consider the $*$-representation
    $$\G \ltimes_{\beta_c} \overline{\mathcal{B}}\to B(\mathcal{H}): \omega \overline{b}\mapsto U(\omega)\rho(b^*).$$
    By (the left version of) Proposition \ref{universalcrossedproduct}, this $*$-representation extends uniquely to a normal, unital $*$-representation $\theta: \G \ltimes_{\beta_c} \overline{B}\to B(\mathcal{H})$. The map
    $$\tilde{\rho}: B \to B(\mathcal{H}): b \mapsto \theta(\beta_c(\overline{b^*}))$$
    is then a unital normal anti-$*$-representation extending $\rho$.
\end{proof}

\begin{Prop}\label{TheoIsoCstar}
Assume $A,B$ are ergodic right $\G$-W$^*$-algebras. Then the surjective map  $$\theta_u: C_u^*(\mathcal{A}\rtimes \G \ltimes \overline{\mathcal{B}})\to C^\G(A,B)$$ from \eqref{StarMorph} is a $*$-isomorphism.
\end{Prop} 
\begin{proof}
Fix a faithful non-degenerate $*$-representation $\theta: C_u^*(\mathcal{A}\rtimes \G\ltimes \overline{\mathcal{B}})\to B(\mathcal{H})$. By Proposition \ref{ErgodicCorrespondence}, there is a $\G$-$A$-$B$-correspondence $(\mathcal{H}, \pi, \rho, U)$ such that $\theta_{\mathcal{H}}\circ \theta_u = \theta$. We conclude that $\theta_u$ is faithful.
\end{proof}

\section{Equivariant amenability and strong equivariant amenability}

Let $\G$ be a locally compact quantum group. Let $A,B$ be $\G$-dynamical von Neumann algebras, and assume we have a $\G$-equivariant normal unital inclusion of von Neumann algebras
\[
A \subseteq B.
\]
Following \cite{DCDR23}, we call $A\subseteq B$ \emph{equivariantly amenable} if there exists a (not necessarily normal) $\G$-equivariant conditional expectation 
\[
E: B \rightarrow A,
\]
and \emph{strongly equivariantly amenable} if ${}_AL^2(A)_A \preccurlyeq {}_AL^2(B)_A$. 

The goal of this section is to prove the following result. For technical reasons, we restrict to the case where $A,B$ are $\sigma$-finite (i.e.\ allow faithful normal states), but we believe this restriction is not essential.

\begin{Theorem}\label{TheoMainAmenab}
Let $\G$ be a compact quantum group. Assume $A \subseteq B$ is an equivariant unital normal inclusion of $\sigma$-finite $\G$-dynamical von Neumann algebras. Then the following are equivalent: 
\begin{enumerate}
\item\label{EqOneDir} $A \subseteq B$ is equivariantly amenable. 
\item\label{EqOtherDir} $A \subseteq B$ is strongly equivariantly amenable. 
\end{enumerate}
\end{Theorem}

The implication \eqref{EqOtherDir} $\Rightarrow$ \eqref{EqOneDir} was already shown in \cite{DCDR23}*{Theorem 4.3} (for general locally compact quantum groups). To prove $\Leftarrow$, we need to set up some machinery.

\subsection{Norm estimates for  \texorpdfstring{$A \odot \overline{A}$}{algebraic tensor products} in the relative setting}

Let $A$ be a von Neumann algebra. Recall that $A$ admits associated non-commutative $L^p$-spaces 
\[
L^p(A),\qquad 1\leq p \leq \infty,
\]
which we concretely realize through the Haagerup construction \cites{Haa79,Ter81}. We follow specifically the exposition of \cite{JP10}*{Chapter 1}, and we assume for now in particular that $A$ admits a faithful normal state $\phi$, to be fixed throughout. Then with $\tau$ the associated semi-finite normal trace on the continuous core 
\[
c_{\phi}(A) = A \rtimes_{\sigma_t^{\phi}}\R,
\]
we can concretely identify $L^p(A)$ as a subspace of the $*$-algebra of $\tau$-measurable operators affiliated to $c_{\phi}(A)$. The norm on $L^p(A)$ is determined through 
\[
\|a\|_p := \trace(|a|^p)^{1/p},
\]
where $\trace: L^1(A) \rightarrow \C$ is the linear functional uniquely determined by
\[
\trace(h) = \omega(1) \textrm{ if }h\in L^1_+(A), \omega \in M_*^+ \textrm{ and }\omega\left(\int_{\R}\hat{\sigma}_t^\phi(x)\rd t\right) = \tau(xh),\qquad x\in c_{\phi}(A)^+,
\]
with $\hat{\sigma}_t^{\phi}$ the dual $\R$-action on $c_{\phi}(A)$. Each $L^p(A)$ is an $A$-bimodule through multiplication.

In particular, $L^2(A)$ will be a Hilbert space. It can be identified with the standard Hilbert space for $A$ as follows: the standard representations $\pi_A$ and $\rho_A$ are through left and right multiplication, while the modular conjugation and standard cone $P^{\natural}$ are defined through 
\[
J_Aa = a^*,\qquad P^{\natural} = \{x^*x\mid x \in L^4(A)\},\qquad a\in L^2(A).
\]
We recall that $\langle \xi,\eta\rangle \geq0$ for all $\xi,\eta\in P^{\natural}$, and that any element of the real span of $P^{\natural}$ can be written in a unique way as a difference of orthogonal elements in $P^{\natural}$. 

 
Let now $A_0 \subseteq A$ be a unital, normal von Neumann subalgebra together with a faithful normal conditional expectation 
\[
P: A \rightarrow A_0 \subseteq A.
\]
We may and will assume that $\phi = \phi \circ P$. We write $\phi_0 = \phi_{\mid A_0}$.

We can identify $L^p(A_0) \subseteq L^p(A)$. If we want to emphasize this inclusion map, we write it as 
\[
I_p: L^p(A_0) \rightarrow L^p(A). 
\]
Similarly, the map $P$ induces a corresponding contractive map between the $L^p$-spaces, which we write  
\[
P_p: L^p(A) \rightarrow L^p(A_0),\qquad a \mapsto P(a).
\]
Here we extend $P$ naturally to a conditional expectation 
\[
P: c_{\phi}(A) \rightarrow c_{\phi}(A_0)
\]
(well-defined as $\phi \circ P = \phi$), which we then in turn may extend to the $*$-algebra of $\tau$-measurable operators.  

With $\xi_{\phi}\in L^2(A_0)$ the GNS-vector of $\phi$, we simply have that 
\[
I_2(\pi_{A_0}(a_0)\xi_{\phi_0})= \pi_A(a_0)\xi_{\phi},\qquad a_0\in A_0,
\]
while $P_2 = I_2^*$, using that $\trace \circ P = \trace$ and 
\[
P(axb) = aP(x)b,\qquad a\in L^p(A_0),b\in L^q(A_0),x\in L^r(A),\quad \frac{1}{p}+\frac{1}{r}+ \frac{1}{q} \leq 1
\]
(see \cite{JP10}*{Lemma 1.4}). 

Consider now on $A^n$ the norms 
\begin{equation}\label{EqNormLeftRightRel}
\|(a_1,\ldots,a_n)\|_0 = \|\frac{1}{n}\sum_i P(a_i^*a_i)\|^{1/2},\qquad 
\|(a_1,\ldots,a_n)\|_1 = \|\frac{1}{n}\sum_i P(a_ia_i^*)\|^{1/2}.
\end{equation}
We denote the respective completions as $L_{\infty}^c(A^n,\widetilde{P})$ and $L_{\infty}^r(A^n,\widetilde{P})$, where we view $A^n$ as a von Neumann algebra equipped with the normal faithful conditional expectation
\[
\widetilde{P}: A^n \rightarrow A_0,\quad (a_1,\ldots,a_n) \mapsto \frac{1}{n}\sum_{k=1}^n P(a_k)
\]
(including $A_0$ into $A^n$ diagonally).

The above Banach spaces can easily be shown to be compatible with respect to the joint subspace $A^n$. We can hence form the complex interpolation space 
\[
[L_{\infty}^c(A^n,\widetilde{P}),L_{\infty}^r(A^n,\widetilde{P})]_{\theta},\qquad 0<\theta<1.
\]
The following result follows from \cite{JP10}*{Corollary 4.10}. The latter (for the case of $\theta =1/2$) extends \cite{Haa93}*{Theorem 2.1} where $A = A_0$, which was in turn an extension of the result of \cite{Pis95}, who proved the statement (for general $\theta$) for $A =A_0$ a semifinite von Neumann algebra.

\begin{Theorem}\label{TheoInterpol}
Let $A,A_0$ and $P: A \rightarrow A_0$ be as above. Then for $x=(x_1, \dots, x_n) \in A^n$, we have
\begin{equation}\label{EqDiffFormRel}
\|(x_1,\ldots,x_n)\|_{1/2} = \|\frac{1}{n}\sum_i P_2\pi_A(x_i)J_A\pi_A(x_i)J_AI_2\|^{1/2}.
\end{equation}
\end{Theorem}
\begin{proof}
This is simply writing out the conclusion of \cite{JP10}*{Corollary 4.10} more explicitly. Namely, we find by that result that
\[
\|(x_1,\ldots,x_n)\|_{1/2} = \sup\{\|axb\|_{L^2(A^n)} \mid \|a\|_{L^4(A_0)},\|b\|_{L^4(A_0)}\leq 1\}. 
\]
But for $b$ fixed, we have 
\begin{eqnarray*}
\sup_{\|a\|_4\leq 1}\|axb\|_{L^2(A^n)}^2 &=& \sup_{\|a\|_4\leq 1} \trace(axbb^*xa^*) \\
&=& \sup_{\|a\|_4\leq 1} \trace(a\widetilde{P}_2(xbb^*x)a^*)\\
&=& \|\widetilde{P}_2(xbb^*x)\|_{L^2(A_0)},
\end{eqnarray*}
using in the last step that $\widetilde{P}_2(xbb^*x)$ is a positive element in $L^2(A_0)$ (see \cite{JP10}*{Lemma 1.4}).

It follows that 
\[
\|(x_1,\ldots,x_n)\|_{1/2}^2 = \sup\{\frac{1}{n}\|\sum_i P_2\pi_A(x_i)J_A\pi_A(x_i)J_AI_2\xi\|\mid \xi \in P_{A_0}^{\natural}, \|\xi\|_2\leq 1\}.
\]
But since $\frac{1}{n}\sum_i P_2\pi_A(x_i)J_A\pi_A(x_i)J_AI_2$ preserves $P_{A_0}^{\natural}$, one easily concludes that \eqref{EqDiffFormRel} must hold.
\end{proof}

\subsection{Conditioned commuting squares} \label{SecCondComSqu}

Assume that we have a square of normal unital inclusions of $\sigma$-finite von Neumann algebras
\begin{equation}\label{EqInclSquares}
\begin{gathered}
\xymatrix{
A \ar[r] & B\\
A_0 \ar[r]\ar[u] & B_0 \ar[u]
}
\end{gathered}
\end{equation}

Assume that we have faithful normal conditional expectations 
\[
P: A\rightarrow A_0,\qquad Q: B\rightarrow B_0,\qquad Q_{\mid A} = P. 
\]
In this case, we say that \eqref{EqInclSquares} is a \emph{conditioned commuting square}.

Assume further that we have a (not necessarily normal) conditional expectation 
\[
E: B \rightarrow A
\]
satisfying the compatibility
\[
E(B_0)  = A_0,\qquad E \circ Q = P\circ E.
\]
In this case, we say that $E$ is a conditional expectation from $(B,B_0,Q)$ to $(A,A_0,P)$. 

The following generalizes partially the harder part of \cite{BMO20}*{Corollorary A.2} (we will not need the direction back for our application). 

\begin{Theorem}\label{TheoBMO}
Assume we are given a conditioned commuting square as in  \eqref{EqInclSquares}. Consider the following statements:
\begin{enumerate}
\item\label{EqCond1} There exists a conditional expectation $E: (B,B_0,Q) \rightarrow (A,A_0,P)$. \item\label{EqCond2} For all algebraic tensors $z = \sum_{i=1}^n  a_i\otimes \overline{a_i} \in A \odot \overline{A}$, we have 
\[
\|\sum_{i=1}^n P_2\pi_A(a_i)\rho_A(a_i^*)P_2^*\|  = \|\sum_{i=1}^n Q_2\pi_B(a_i)\rho_B(a_i^*)Q_2^*\|.
\] 
\item\label{EqCond3} For all algebraic tensors $z = \sum_{i=1}^n a_i'\otimes \overline{a_i} \in A \odot \overline{A}$ and all normal states $\omega \in A_*$ satisfying $\omega = \omega \circ P$, we have 
\[
| \langle\xi_{\omega},\sum_{i=1}^n \pi_A(a_i')\rho_A(a_i^*)\xi_{\omega}\rangle| \leq \|\sum_{i=1}^n Q_2\pi_B(a_i')\rho_B(a_i^*)Q_2^*\|.
\] 
\end{enumerate}
Then $\eqref{EqCond1}  \Rightarrow  \eqref{EqCond2} \Rightarrow \eqref{EqCond3}$. 
\end{Theorem} 
\begin{proof}
Assume \eqref{EqCond1}. Since $E$ intertwines $Q$ and $P$, we see by the Schwarz inequality that $E$ is contractive with respect to both norms in \eqref{EqNormLeftRightRel}. Hence $E$ must also be contractive with respect to the interpolated norm $\|-\|_{1/2}$. By Theorem \ref{TheoInterpol}, this entails that $\leq$ holds in \eqref{EqCond2}. The reverse inequality holds by applying the same reasoning to the inclusion map $A \rightarrow B$. 

Assume now that \eqref{EqCond2} holds. Then we can simply copy the proof of \cite{BMO20}*{Corollorary A.2} to show that \eqref{EqCond3} holds. We spell out the details for the minor modifications that need to be made.

Take any normal state $\omega_0$ on $A_0$, and put $\omega =  \omega_0\circ P$. Let 
\[
s_{\omega}: A \times A \rightarrow \C,\qquad s_{\omega}(a',a) = \langle \xi_{\omega},\pi_A(a')\rho_A(a^*)\xi_{\omega}\rangle
\]
be the associated self-polar form, which is anti-linear in the \emph{second} variable. Since $P_2^*\xi_{\omega_0} = \xi_{\omega}$, we get by our assumption that
\[
\sum_{i=1}^n s_{\omega}(a_i,a_i) = \langle \xi_{\omega_0}, P_2\sum_{i=1}^n \pi_A(a_i)\rho_A(a_i^*)P_2^*\xi_{\omega_0}\rangle \leq\| Q_2\sum_{i=1}^n \pi_B(a_i) \rho_B(a_i^*)Q_2^*\|,\qquad a_1,\ldots,a_n \in A.
\]

So considering the convex cone 
\[
C := \{Q_2\sum_{i=1}^n \pi_B(a_i) \rho_B(a_i^*)Q_2^* \mid n\in\N, a_i \in A\} \subseteq B(L^2(B_0)),
\]
we can define a superlinear function 
\[
q: C \rightarrow \R,\quad Q_2\sum_{i=1}^n \pi_B(a_i) \rho_B(a_i^*)Q_2^* \mapsto \sum_{i=1}^n s_{\omega}(a_i,a_i). 
\]
Using \cite{BMO20}*{Lemma A.5} and using the same argument as in the proof of \cite{BMO20}*{Corollary A.2}, this entails that there exists a state $\psi$ on $B(L^2(B_0))$ with 
\[
\sum_i s_{\omega}(a_i,a_i)  \leq \mathrm{Re}\left(\psi(Q_2\sum_i \pi_B(a_i) \rho_B(a_i^*)Q_2^*)\right).
\]
Since $Q_2$ intertwines $J_{B_0}$ with $J_B$, we can replace $\psi$ by $\frac{1}{2}(\psi + \psi(J_{B_0}(-)^*J_{B_0}))$ to get that 
\[
s: B \times B \rightarrow \C,\qquad (b',b) \mapsto \psi(Q_2 \pi_B(b') \rho_B(b^*)Q_2^*)
\]
is a sesqui-linear hermitian form with $s_{\omega}(a,a)\leq s(a,a)$ for all $a\in A$. Hence $s_{\omega} = s$ by \cite{Wor74}*{Theorem 1.1 and Theorem 2.1}, leading to 
\[
\langle \xi_{\omega}, \pi_A(a')\rho_A(a^*)\xi_{\omega}\rangle = \psi(Q_2\pi_B(a')\rho_B(a^*)Q_2^*),\qquad a,a'\in A.
\]
So \eqref{EqCond3} follows. 
\end{proof}

\subsection{Proof of Theorem \ref{TheoMainAmenab}}

Assume now again that $\G = (M,\Delta)$ is a CQG, and that $A,B$ are $\sigma$-finite $\G$-W$^*$-algebras. We write then by $P,Q$ the associated faithful normal conditional expectations 
\[
P: A \rightarrow A_0 = A^{\G}, \quad a \mapsto (\id\otimes \Phi)\alpha(a),\qquad Q: B \rightarrow B_0 = B^{\G}, \quad b \mapsto (\id\otimes \Phi)\beta(b).
\]
Recalling the notation \eqref{EqDefLowerTheta} and \eqref{EqDefTrivRep}, we have in this case that
\[
P_2^* \circ P_2  = \theta_{L^2(A)}(p_{\varepsilon}),
\]
and similarly for $Q$.

Assume now that $A \subseteq B$ is an equivariant normal inclusion. Then in particular \eqref{EqInclSquares} is a conditioned square. 

We can now complete the proof of Theorem \ref{TheoMainAmenab}. Indeed, assume that $A \subseteq B$ is equivariantly amenable. As in \eqref{EqRestrRep}, we then get a $*$-representation $\theta^{L^2(B)}_{\square}$ of $\mcA \square\overline{\mcA}$ on $L^2(B_0)$ by 
\[
\theta^{L^2(B)}_{\square}(\sum_i a_i' \otimes \overline{a_i}) = \sum_i Q_2\pi_B(a_i')\rho_B(\delta^{1/2}(a_{i(1)}^*)a_{i(0)}^*)Q_2^*,
\]
and we know by Proposition \ref{weak containment} that 
\[
{}_AL^2(A)_A \preccurlyeq {}_AL^2(B)_A \textrm{ equivariantly }\quad \Leftrightarrow  \quad \theta_{\square}^{L^2(A)} \preccurlyeq \theta_{\square}^{L^2(B)}.
\]
But by the conclusion of Theorem \ref{TheoBMO}, we have 
\[
|\langle \xi,\theta_{\square}^{L^2(A)}(z)\xi\rangle| \leq \|\theta_{\square}^{L^2(B)}(z)\|,\qquad \forall z\in \mcA\square \overline{\mcA},\forall \xi\in P^{\natural}_{A_0}, \|\xi\|=1.
\]
Since $P^{\natural}_{A_0}$ contains the $\pi_{A_0}$-cyclic vector $\xi_{\phi_0}$, and hence a fortiori a $\theta^{L^2(A)}_{\square}$-cyclic vector, we conclude that 
\begin{equation}\label{EqNormEstSquare}
\|\theta^{L^2(A)}_{\square}(z)\|\leq \|\theta^{L^2(B)}_{\square}(z)\|,\qquad \forall z\in \mcA\square\overline{\mcA},
\end{equation}
finishing the proof.

We now prove the analogue of Theorem \ref{TheoMainAmenab} for discrete quantum groups. First, we observe the following general result:

\begin{Prop}\label{crossed}
    Let $\G$ be a locally compact quantum group and $A\subseteq B$ a normal, unital $\G$-inclusion of $\G$-$W^*$-algebras. The following are equivalent:
    \begin{enumerate}
        \item $A\subseteq B$ is (strongly) equivariantly $\G$-amenable.
        \item $A\rtimes \G\subseteq B\rtimes \G$ is (strongly) equivariantly $\check{\G}$-amenable.
    \end{enumerate}
\end{Prop}
\begin{proof}
    In the non-strong setting, this is proven in \cite{DR24}*{Proposition 2.6}. In the strong setting, this follows by combining \cite{DCDR23}*{Proposition 5.27, Proposition 5.28 and Theorem 6.4}.
\end{proof}

\begin{Theorem}\label{discretemain}
    Let $\mathbbl{\Gamma}$ a discrete quantum group and assume $A \subseteq B$ is an equivariant unital normal inclusion of $\sigma$-finite $\mathbbl{\Gamma}$-dynamical von Neumann algebras. Then the following are equivalent: 
\begin{enumerate}
\item $A \subseteq B$ is equivariantly amenable. 
\item $A \subseteq B$ is strongly equivariantly amenable. 
\end{enumerate}
\end{Theorem}
\begin{proof}  Consider the canonical faithful ucp normal conditional expectation 
$$B \rtimes_\beta \mathbbl{\Gamma} \to \beta(B)\cong B.$$
Composing with a faithful normal state on $B$, we see that $B\rtimes_\beta \mathbbl{\Gamma}$ is $\sigma$-finite as well. We then argue:
\begin{align*}
    A\subseteq B \textrm{\ is\ } \mathbbl{\Gamma} \textrm{-equivariantly amenable} &\stackrel{\textrm{Prop.} \ref{crossed}}\iff A\rtimes \mathbbl{\Gamma}\subseteq B\rtimes \mathbbl{\Gamma} \textrm{\ is\ } \check{\mathbbl{\Gamma}} \textrm{-equivariantly amenable} \\
    &\stackrel{\textrm{Thm.} \ref{TheoMainAmenab}}\iff A\rtimes \mathbbl{\Gamma}\subseteq B\rtimes \mathbbl{\Gamma} \textrm{\ is\ strongly\ } \check{\mathbbl{\Gamma}} \textrm{-equivariantly amenable} \\
       &\stackrel{\textrm{Prop.} \ref{crossed}}\iff A\subseteq B \textrm{\ is\ strongly } \mathbbl{\Gamma} \textrm{-equivariantly amenable}.
\end{align*}
The result follows.
\end{proof}

Let us give some concrete applications of the above results:

\begin{Theorem}\label{application}
     Let $\G$ be a compact quantum group and $(A, \alpha)$ be a $\sigma$-finite $\G$-$W^*$-algebra. The following statements hold:
    \begin{enumerate}
        \item $(A, \alpha)$ is strongly amenable if and only if $(A, \alpha)$ is amenable.
        \item $(A, \alpha)$ is strongly inner amenable if and only if $(A, \alpha)$ is inner amenable.
    \end{enumerate}
\end{Theorem}
\begin{proof}
By definition, we have that  \begin{itemize}
            \item  $(A, \alpha)$ is (strongly) amenable iff the $\G$-inclusion $\alpha(A) \subseteq A \bar{\otimes}L^\infty(\G)$ is (strongly) equivariantly amenable.
        \item $(A, \alpha)$ is (strongly) inner amenable iff the $\G$-inclusion $\alpha(A)\subseteq A\rtimes_\alpha \G$ is (strongly) equivariantly amenable.
    \end{itemize}

Hence if $A$ is $\sigma$-finite, the first bullet point in Theorem \ref{application} follows immediately from Theorem \ref{TheoMainAmenab} and the observation that also $A\bar{\otimes}L^\infty(\G)$ is $\sigma$-finite.

For the second bullet point of Theorem \ref{application}, we can conclude similarly in case $\mcO(\G)$ has countable dimension, since then again $A\rtimes_\alpha \G$ is $\sigma$-finite. 

For the general case, we proceed as follows. As this concerns only a minor technical improvement, we will be rather brief. 

If $\mcO(\Hh) \subseteq \mcO(\G)$ is any Hopf $*$-subalgebra, let us write $A(\Hh) \subseteq A$ for the unital von Neumann subalgebra spanned by the $\Irr(\Hh)$-spectral subspaces of $A$. Then  $A(\Hh)$ is a $\G$-equivariant von Neumann subalgebra, endowed with a $\G$-equivariant normal conditional expectation $A\rightarrow A(\Hh)$ (killing all spectral subspaces at irreducible representations of $\Irr(\G) \setminus \Irr(\Hh)$). In particular, $L^2(A(\Hh))\subseteq L^2(A)$ as a $\G$-equivariant $A(\Hh)$-$A(\Hh)$-correspondence. We also recall that the unitary antipode of $\mcO(\G)$ restricts to the unitary antipode of $\mcO(\Hh)$. 

Fix now some $z = \sum_i a_i' \otimes \Phi(-g_i) \otimes \overline{a_i}$ in $\mcA\rtimes \G \ltimes \overline{\mcA}$. Then we can find a Hopf $*$-subalgebra $\mcO(\Hh) \subseteq \mcO(\G)$ of countable dimension such that 
\begin{itemize}
\item $\{\alpha(a_i),\alpha(a_i')\} \subseteq \mcA \odot \mcO(\Hh)$,
\item $\{g_i\}\subseteq \mcO(\Hh)$, and 
\item $\|\theta^{L^2(A)}(z)\| = \|\theta^{L^2(A)}(z)_{\mid L^2(A(\Hh))}\| = \|\theta^{L^2(A(\Hh))}(z)\|$.
\end{itemize}

Now our given $\G$-equivariant conditional expectation $E: A\rtimes_{\alpha}\G\rightarrow A$ will restrict to an $\Hh$-equivariant conditional expectation $E': A(\Hh)\rtimes_{\alpha}\Hh \rightarrow A(\Hh)$. For this case, we can then conclude that 
\begin{equation}\label{EqPrelimEst}
\|\theta^{L^2(A)}(z)\|  = \|\theta^{L^2(A(\Hh))}(z)\| \leq \|\theta^{L^2(A(\Hh)\rtimes_{\alpha}\Hh)}(z)\|.
\end{equation}
However, it is also easy to see that we have a $\G$-equivariant normal condition expectation from $A\rtimes \G$ onto the (non-unital) von Neumann subalgebra $A(\Hh)\rtimes \Hh$, factorizing as
\[
A\rtimes \G\rightarrow A(\Hh)\rtimes \G\rightarrow A(\Hh) \rtimes \Hh.
\]
Indeed, it suffices to note that there is a canonical normal surjective $*$-morphism $\check{\kappa}: \ell^\infty(\check{\G})\to \ell^\infty(\check{\mathbb{H}})$. An argument similar as in the proof of \cite{DR24}*{Proposition 8.1} then shows that there is a normal conditional expectation
$$A(\mathbb{H})\rtimes \G \to A(\mathbb{H})\rtimes \mathbb{H}$$
which maps
$$\alpha(a)\mapsto \alpha(a), \quad 1 \otimes \check{y}\mapsto 1\otimes \check{\kappa}(\check{y}), \quad a \in A(\mathbb{H}), \quad \check{y}\in \ell^\infty(\check{\mathbb{G}}).$$
It hence follows from Lemma \ref{LemEstimateNorm} that we can also write 
\[
\|\theta^{L^2(A(\Hh)\rtimes_{\alpha}\Hh)}(z)\|\leq \|\theta^{L^2(A\rtimes_{\alpha}\G)}(z)\|.
\]
Combining with \eqref{EqPrelimEst}, we conclude that 
\[
\|\theta^{L^2(A)}(z)\| \leq \|\theta^{L^2(A\rtimes_{\alpha}\G)}(z)\|.
\]
As $z$ was arbitrary, this shows that $L^2(A)$ is $\G$-equivariantly weakly contained in $L^2(A\rtimes \G)$, i.e.\ $A$ is strongly inner amenable.
\end{proof}

Note that the $\sigma$-finiteness assumption on $A$ in the previous theorem is satisfied as soon as $A^{\G}$ is $\sigma$-finite, e.g.\ when $(A,\alpha)$ is ergodic. 

By duality, we then also obtain the following result:
\begin{Theorem}\label{application2}
    Let $\mathbbl{\Gamma}$ be a discrete quantum group, and let $(A, \alpha)$ be a $\sigma$-finite $\mathbbl{\Gamma}$-$W^*$-algebra. The following statements hold:
    \begin{enumerate}
        \item  $(A, \alpha)$ is strongly amenable if and only if $(A, \alpha)$ is amenable.
        \item $(A, \alpha)$ is strongly inner amenable if and only if $(A, \alpha)$ is inner amenable.
    \end{enumerate}

    In particular, $\mathbbl{\Gamma}$ is strongly inner amenable if and only if $\mathbbl{\Gamma}$ is inner amenable.

\end{Theorem}
\begin{proof} The equivalences $(1)$ and $(2)$ immediately follow from Theorem \ref{application} and \cite{DCDR23}*{Theorem 6.12}. The equivalence between strong inner amenability and inner amenability for $\mathbbl{\Gamma}$ follows by applying $(2)$ to the trivial action $\mathbb{C}\curvearrowleft \mathbbl{\Gamma}$.
\end{proof}

\begin{Rem}
    Taking $A= \mathbb{C}$ in Theorem \ref{application2}.(1), we see that $\mathbbl{\Gamma}$ is amenable if and only if the compact dual $\check{\mathbbl{\Gamma}}$ is co-amenable, which is Tomatsu's celebrated result \cite{Tom06}. Therefore, Theorem \ref{application2} can be seen as a dynamical generalization of Tomatsu's result.
\end{Rem}

\section{Examples of non-amenable actions}

In this section, we provide examples of non-amenable actions of co-amenable compact and amenable discrete quantum groups. We start with some general remarks:

\begin{itemize}
\item If $\G$ is a locally compact quantum group, then the translation action of $\G$ on itself is amenable if and only if the dual quantum group $\check{\G}$ is inner amenable, and this holds automatically if $\G$ is co-amenable (\cite{DR24}*{Theorem 6.4}, \cite{Cr19}*{Corollary 3.9}). Using this, it is not so hard to find examples of non-amenable actions of non-co-amenable compact quantum groups (e.g.\ the dual of $A_o(F)$ is a discrete quantum group that is not inner amenable for an appropriate matrix $F \in GL_n(\mathbb{C})$).
\item If $\G$ is a compact quantum group of Kac type, then every action of $\G$ is amenable, as easily follows by combining \cite{DR24}*{Lemma 6.12} and \cite{DCDR23}*{Theorem 6.12}. It is also possible to prove this directly using modular theory, as communicated to us by S. Vaes.
\item If $\G$ is a compact quantum group and $\alpha: A\curvearrowleft \G$ an action that is not amenable, then the induced action $A\rtimes_\alpha \G \curvearrowleft \check{\G}$ is not inner amenable \cite{DCDR23}*{Theorem 6.12} and a fortiori not amenable \cite{DR24}*{Proposition 6.9}. Thus, any example of a co-amenable compact quantum group acting non-amenably on a von Neumann algebra gives an example of an amenable discrete quantum group acting non-amenably on a von Neumann algebra.
\end{itemize}

We now provide two strategies that can be useful to detect (non-)amenability of a given action of a compact quantum group on a von Neumann algebra, and apply these strategies to concrete examples. This will yield the first examples of co-amenable compact quantum groups acting non-amenably on von Neumann algebras (and by duality, also the first examples of amenable discrete quantum groups acting non-amenably on von Neumann algebras).

\subsection{First strategy} Let $\G=(M,\Delta)$ be a compact quantum group and $(A, \alpha)$ be a $\G$-$W^*$-algebra. The semi-coarse $\G$-$A$-$A$ correspondence $S_A^\G$ \cite{DCDR23}*{Example 2.6} is given by the following data:
\begin{itemize}
    \item The Hilbert space $L^2(A)\otimes L^2(\G).$
    \item The normal $*$-representation $\pi_S:= (\pi_A\otimes \id)\circ \alpha: A \to B(L^2(A)\otimes L^2(\G))$.
    \item The normal anti-$*$-representation $\rho_S:= (\rho_A\otimes \rho_M)\circ \alpha: A \to B(L^2(A)\otimes L^2(\G))$.
    \item The unitary $\G$-representation $U_S := V_{23}\in B(L^2(A)\otimes  L^2(\G))\bar{\otimes} M.$
\end{itemize}
We recall from \cite{DCDR23} that $(A, \alpha)$ is strongly $\G$-amenable (and hence $\G$-amenable) if and only if $L^2(A) \preccurlyeq S_A^\G$ equivariantly. By Proposition \ref{weak containment}, this is equivalent with the weak containment
\begin{equation}\label{weak}
    \epsilon_A:=\theta^{L^2(A)}_{\square}\preccurlyeq \theta^{S^\G_A}_\square=: \theta_S.
\end{equation}

In concrete examples, the aim is to  describe explicitly the cotensor product $\mathcal{A}\square \overline{\mathcal{A}}$ and the $*$-representations $\epsilon_A, \theta_S$ on $\mathcal{A}\square \overline{\mathcal{A}}$ in a way that allows us to detect if the weak containment \eqref{weak} holds. 

We illustrate this strategy with a concrete example. More concretely, if  $(A, \alpha)$ is an ergodic right $\G$-$W^*$-algebra, i.e. $A^\G = \C 1$, then we show that the two $*$-representations $\epsilon_A$ and $\theta_S$ admit a more convenient description (Proposition \ref{concrete}). In a next step, we will choose a specific ergodic $\G$-$W^*$-algebra $(A, \alpha)$ such that the cotensor product $\mathcal{A} \square\overline{\mathcal{A}}$ becomes a concrete well-understood $*$-algebra (Proposition \ref{Prop}) on which the $*$-representations $\epsilon_A$ and $\theta_S$ can be explicitly calculated (Proposition \ref{semi}). Concrete knowledge of these $*$-representations will then allow us to characterize when the weak containment \eqref{weak} holds (Theorem \ref{TheoMain2}).

Let us first recall some useful facts that follow from the ergodicity of $(A,\alpha)$. There is a unique normal $\G$-equivariant state $\Phi_A: A \to \mathbb{C}$, explicitly given by $\Phi_A = (\id \otimes \Phi)\circ \alpha: A \to A^\G \cong \mathbb{C}$. It is faithful, and we can thus consider the standard Hilbert space $L^2(A)$ and associated modular group  $\{\sigma_t^A\}_{t\in \mathbb{R}}$ with respect to this state. The algebraic core $\mathcal{A}$ is analytic for this modular group, and the associated canonical unitary $\G$-representation is concretely implemented by \eqref{action of dual}.

Note now that \eqref{action of dual} implies that $\theta^{L^2(A)}(p_\varepsilon)\Lambda_A(a) = \Lambda_A(1)\Phi_A(a)$ for $a\in \mathcal{A}$. Therefore, 
\begin{equation}\label{reference}
    \theta^{L^2(A)}(p_\varepsilon) L^2(A) = \mathbb{C}\Lambda_A(1)\cong \mathbb{C},
\end{equation}
and we can view
$\epsilon_A= \theta_\square^{L^2(A)}$ as a $*$-character $\mathcal{A} \overset{\G}{\square}\overline{\mathcal{A}} \to \mathbb{C}.$ Similarly, note that 
$$V_{23}(p_\varepsilon)(\Lambda_A(a)\otimes \Lambda(x)) = \Lambda_A(a)\otimes p_\varepsilon(x)\Lambda(1), \quad a \in \mathcal{A}, \quad x \in \mathcal{O}(\G).$$
Therefore, $\theta^{S_A^\G}(p_\varepsilon)(L^2(A)\otimes L^2(\G)) = L^2(A)\otimes \mathbb{C}\Lambda(1)\cong L^2(A)$ and we can view $\theta_S=\theta_\square^{S_A^\G}$ as a $*$-representation $\mathcal{A} \overset{\G}{\square}\overline{\mathcal{A}} \to B(L^2(A)).$

It is desirable to have concrete formulas for these $*$-representations. For this, it is convenient to introduce the following notation
$$\kappa_A: \mathcal{A}\to \mathcal{A}: a \mapsto \sigma_{-i/2}^A(a_{(0)})\delta^{-1/2}(a_{(1)}).$$

\begin{Prop}\label{concrete} Let $(A, \alpha)$ be an ergodic $\G$-$W^*$-algebra.\begin{enumerate} 
    \item The $*$-representation $\epsilon_A: \mathcal{A} \overset{\G}{\square}\overline{\mathcal{A}}\to \mathbb{C}$ is given by 
    $$\epsilon_A(\sum_{j=1}^n a_j \otimes \overline{b_j}) = \sum_{j=1}^n \Phi_A(b_j^*\kappa_{A}(a_j)).$$
    \item  The $*$-representation $\theta_S: \mathcal{A} \overset{\G}{\square}\overline{\mathcal{A}}\to B(L^2(A))$ is given by
    \begin{align*}
        \theta_S(\sum_{j=1}^n a_j\otimes \overline{b_j})= \sum_{j=1}^n \pi_A(a_{j,(0)}) \rho_A(b_{j,(0)}^*) \Phi_\G(b_{j,(1)}^* \kappa_{L^\infty(\G)}(a_{j,(1)})).
    \end{align*}
\end{enumerate}
  
\end{Prop}\begin{proof} Observe that 
\begin{align}\label{cotensor}
    \sum_{j=1}^n a_j\otimes \overline{b_j}\in \mathcal{A} \overset{\G}{\square} \overline{\mathcal{A}}&\iff \sum_{j=1}^n a_{j,(0)}\otimes R(a_{j,(1)})\otimes b_j^* = \sum_{j=1}^n a_j\otimes b_{j,(1)}^*\otimes b_{j,(0)}^*.
\end{align}

(1) Given $\sum_j a_j\otimes \overline{b_j}\in \mathcal{A}\overset{\G}{\square} \overline{\mathcal{A}}$, we find
\begin{align*}
    \theta^{L^2(A)}_\square(\sum_j a_j\otimes \overline{b_j}) \Lambda_A(1)&= \sum_j \Lambda_A(a_j \sigma_{-i/2}^A(b_{j,(0)}^*)) \delta^{1/2}(b_{j,(1)}^*)\\
    &\stackrel{\eqref{cotensor}}= \sum_j \Lambda_A(a_{j,(0)} \delta^{1/2}(R(a_{j,(1)})) \sigma_{-i/2}^A(b_j^*))= \sum_j \Lambda_A(a_{j,(0)} \sigma_{-i/2}^A(b_j^*)) \delta^{-1/2}(a_{j,(1)}).
\end{align*}
Since $\sum_j \Lambda_A(a_{j,(0)} \sigma_{-i/2}^A(b_j^*))\delta^{-1/2}(a_{j,(1)}) \in \mathbb{C} \Lambda_A(1)$ (see \eqref{reference}), it follows (take inner products with respect to the vector $\Lambda_A(1)$) that
\begin{align*}\sum_j \Lambda_A(a_{j,(0)} \sigma_{-i/2}^A(b_j^*))\delta^{-1/2}(a_{j,(1)}) &= \sum_j \Phi_A(a_{j,(0)}\sigma_{-i/2}^A(b_j^*))\delta^{-1/2}(a_{j,(1)})\Lambda_A(1).
\end{align*}
Consequently,
\begin{align*}
    \theta_{\square}^{L^2(A)}(\sum_j a_j\otimes \overline{b_j})\Lambda_A(1) &= \sum_{j} \Phi_A(b_j^* \sigma_{-i/2}^A(a_{j,(0)})) \delta^{-1/2}(a_{j,(1)}) \Lambda_A(1)= \sum_j \Phi_A(b_j^* \kappa_A(a_j))\Lambda_A(1),
\end{align*}
so we deduce that $$\epsilon_A(\sum_j a_j\otimes \overline{b_j})= \sum_j \Phi_A(b_j^*\kappa_A(a_j)).$$
(2) If $a,b,c\in \mathcal{A}$, it is easily verified that
\begin{equation}\label{eq}
    \pi_S(a)\delta^{1/2}(b_{(1)}^*) \rho_S(b_{(0)}^*) (\Lambda_A(c)\otimes \Lambda(1)) = \Lambda_A(a_{(0)}c\sigma_{-i/2}^A(b_{(0)}^*)) \otimes \Lambda(a_{(1)} \sigma_{-i/2}(b_{(1)}^*)) \delta^{1/2}(b_{(2)}^*).
\end{equation}
If $\sum_j a_j\otimes \overline{b_j}\in \mathcal{A} \overset{\G}{\square}\overline{\mathcal{A}}$, \eqref{cotensor} implies that also
\begin{align}\label{cotensor2}
    \sum_j a_{j,(0)}\otimes a_{j,(1)} \otimes R(a_{j,(2)})\otimes b_{j,(0)}^*\otimes b_{j,(1)}^* = \sum_j a_{j,(0)}\otimes a_{j,(1)}\otimes b_{j,(2)}^*\otimes b_{j,(0)}^*\otimes b_{j,(1)}^*,
\end{align}
and thus by combining \eqref{eq} and \eqref{cotensor2}, we find
\begin{align*}
    \theta_{\square}^{S_A^\G}(\sum_j a_j\otimes \overline{b_j})(\Lambda_A(c)\otimes \Lambda(1))&= \sum_j \Lambda_A(a_{j,(0)}c\sigma_{-i/2}^A(b_{j,(0)}^*))\otimes \delta^{1/2}(R(a_{j,(2)})) \Lambda(a_{j,(1)} \sigma_{-i/2}(b_{j,(1)}^*))\\
    &= \sum_j \Lambda_A(a_{j,(0)}c\sigma_{-i/2}^A(b_{j,(0)}^*))\otimes \delta^{-1/2}(a_{j,(2)}) \Lambda(a_{j,(1)}\sigma_{-i/2}(b_{j,(1)}^*)).
\end{align*}

    Therefore,
$$\theta_S(\sum_{j=1}^n a_j\otimes \overline{b_j})=\sum_{j=1}^n \pi_A(a_{j,(0)}) \rho_A(b_{j,(0)}^*) \Phi_\G(b_{j,(1)}^* \kappa_{L^\infty(\G)}(a_{j,(1)})).$$
\end{proof}

Fix a compact quantun group $\mathbb{H}$ and consider the compact quantum group $\G := \mathbb{H}^{\op}\times \mathbb{H}$. Concretely,
$$L^\infty(\G) = L^\infty(\mathbb{H})\bar{\otimes}L^\infty(\mathbb{H}), \quad \Delta_\G = (\id \otimes \varsigma \otimes \id) \circ (\Delta_{\mathbb{H}}^{\op}\otimes \Delta_{\mathbb{H}}),$$
with $\varsigma$ the flip. Then with $\Delta^{(2)} = (\Delta\otimes \id)\Delta$, it is easily verified that
$$\alpha: L^\infty(\mathbb{H})\to L^\infty(\mathbb{H})\bar{\otimes} L^\infty(\G) = L^\infty(\mathbb{H})^{\bar{\otimes} 3}: x \mapsto \Delta_{\mathbb{H}}^{(2)}(x)_{213}$$
defines an ergodic right action $L^\infty(\mathbb{H})\curvearrowleft \G$. On the algebraic regular elements, the coaction restricts to $$\alpha: \mathcal{O}(\mathbb{H})\to \mathcal{O}(\mathbb{H})\odot \mathcal{O}(\G): x\mapsto x_{(2)}\otimes (x_{(1)}\otimes x_{(3)}).$$
To make the connection with the previous notations, write
$$A:= L^\infty(\mathbb{H}), \quad \mathcal{A} := \mathcal{O}(\mathbb{H}).$$
  The associated left action $\G\curvearrowright \overline{\mathcal{O}(\mathbb{H})}$ is given by
$$\alpha_c: \overline{\mathcal{O}(\mathbb{H})}\to \mathcal{O}(\mathbb{G})\odot \overline{\mathcal{O}(\mathbb{H})}: \overline{x} \mapsto R(x_{(1)}^*)\otimes R(x_{(3)}^*)\otimes \overline{x_{(2)}}.$$
Given $\pi\in \operatorname{Irr}(\mathbb{H})$, we recall the notation $n_\pi = \dim(\pi) = \dim(\mathcal{H}_\pi)$. Consider then the character
$$\chi(\pi):= \sum_{i=1}^{n_\pi} U_\pi(e_i^\pi, e_i^\pi)\in \mathcal{O}(\mathbb{H}).$$
The elements $\{\chi(\pi)\}_{\pi\in \operatorname{Irr}(\mathbb{H})}$ are linearly independent. We write 
$$\operatorname{Fus}[\mathbb{H}]= \operatorname{span}\{\chi(\pi): \pi\in \Irr(\mathbb{H})\}= \{x\in \mathcal{O}(\mathbb{H}): \Delta(x)= \Delta^{\op}(x)\}$$
for the \emph{fusion algebra} of $\mathbb{H}$. It is a $*$-subalgebra of $\mathcal{O}(\mathbb{H})$. 

\begin{Prop}\label{Prop}
    The map $$\operatorname{Fus}[\mathbb{H}]\cong \mathcal{O}(\mathbb{H}) \overset{\G}{\square} \overline{\mathcal{O}(\mathbb{H})}: \chi(\pi) \mapsto \sum_{j,k=1}^{n_\pi} U_\pi(e_j^\pi, e_k^\pi)\otimes \overline{U_{\overline{\pi}}(\overline{e_j^\pi}, \overline{e_k^\pi})^*}$$
    is an isomorphism of unital $*$-algebras.
\end{Prop}
\begin{proof} 
We have a $*$-isomorphism
$$\overline{\mathcal{O}(\mathbb{H})} \cong \mathcal{O}(\mathbb{H}): \overline{x}\mapsto R(x^*),$$
and with respect to this isomorphism, we get an induced action $\G\curvearrowright \mathcal{O}(\mathbb{H})$ given by
$$\beta: \mathcal{O}(\mathbb{H})\to \mathcal{O}(\mathbb{G})\odot \mathcal{O}(\mathbb{H}): x \mapsto (x_{(3)}\otimes x_{(1)})\otimes x_{(2)}.$$
We then obtain an induced isomorphism
$ \mathcal{O}(\mathbb{H}) \overset{\G}{\square} \overline{\mathcal{O}(\mathbb{H})}\cong  \mathcal{O}(\mathbb{H})  \overset{\G}{\square} \mathcal{O}(\mathbb{H}).$

We claim that $\Delta_\mathbb{H}(\operatorname{Fus}[\mathbb{H}])\subseteq  \mathcal{O}(\mathbb{H}) \overset{\G}{\square}\mathcal{O}(\mathbb{H})$. Indeed, if $z\in \operatorname{Fus}[\mathbb{H}]$, then 
\begin{align*}
    &(\alpha \odot \id)\Delta_{\Hh}(z)=z_{(2)}\otimes z_{(1)}\otimes z_{(3)}\otimes z_{(4)},\\
    &(\id \odot \beta)\Delta_{\Hh}(z)= z_{(1)}\otimes z_{(4)}\otimes z_{(2)}\otimes z_{(3)}. 
\end{align*}
Since $\Delta_{\Hh}(z) = \Delta_{\Hh}^{\op}(z)$, these expressions are equal. Conversely, let $z\in \mathcal{O}(\mathbb{H}) \overset{\G}{\square} \mathcal{O}(\mathbb{H})$, i.e. 
$$(\alpha \odot \id)(z) = (\id \odot \beta)(z).$$
Applying $\id \odot \varepsilon \odot \id \odot \id$ to this expression, we get $(\Delta_{\Hh}\odot \id)(z) = (\id \odot \Delta_\Hh)(z)$. On the other hand, applying $\id \odot \id \odot \varepsilon \odot \id$ to the same expression yields $(\Delta_{\Hh}^{\op}\odot \id)(z) = (\id \odot \Delta_\Hh^{\op})(z)$. Put $w:= (\id \odot \varepsilon)(z)\in \mathcal{O}(\Hh)$. We see that $z=\Delta_\Hh(w) = \Delta_{\Hh}^{\op}(w)$, and thus $w \in \operatorname{Fus}[\Hh]$. It follows that the maps
$$
\begin{tikzcd}
{\operatorname{Fus}[\Hh]} \arrow[rrr, "\Delta_\Hh", shift left] &  &  & \mathcal{O}(\mathbb{H}) \overset{\G}{\square} \mathcal{O}(\mathbb{H}) \arrow[lll, "\id \odot \varepsilon_\Hh", shift left]
\end{tikzcd}$$
are inverse to each other. Composing the isomorphisms
$$\operatorname{Fus}[\Hh]\cong  \mathcal{O}(\mathbb{H}) \overset{\G}{\square}\mathcal{O}(\mathbb{H}) \cong \mathcal{O}(\mathbb{H}) \overset{\G}{\square} \overline{\mathcal{O}(\mathbb{H})},$$
we obtain the map
$$\chi(\pi)\mapsto \sum_{j,k=1}^{n_\pi} U_\pi(e_j^\pi, e_k^\pi)\otimes \overline{R(U_\pi(e_k^\pi, e_j^\pi)^*)}= \sum_{j,k=1}^{n_\pi} U_\pi(e_j^\pi, e_k^\pi)\otimes \overline{U_{\overline{\pi}}(\overline{e_j^\pi}, \overline{e_k^\pi})^*}.$$
\end{proof}

\begin{Prop} \label{semi}\begin{enumerate}
\item The $*$-representation $\epsilon_A: \operatorname{Fus}[\mathbb{H}]\cong \mathcal{O}(\mathbb{H}) \overset{\G}{\square} \overline{\mathcal{O}(\mathbb{H})}\to \mathbb{C}$ is given by
$$\epsilon_A(\chi(\pi))= \dim_q(\pi), \quad \pi\in \operatorname{Irr}(\mathbb{H}).$$
    \item The $*$-representation $\theta_S: \operatorname{Fus}[\mathbb{H}]\cong  \mathcal{O}(\mathbb{H}) \overset{\G}{\square} \overline{\mathcal{O}(\mathbb{H})}\to B(L^2(\mathbb{H}))$ coincides with the restriction of the $*$-representation
$$\mathcal{O}(\mathbb{H})\to B(L^2(\mathbb{H})): x \mapsto \pi_{\mathbb{H}}(x_{(1)}) \rho_{\mathbb{H}}(R(x_{(2)})).$$
\end{enumerate}
\end{Prop}
\begin{proof}
 (1) We note that, e.g.\ by the formula \ref{DefWorChar}, we have the equality of Woronowicz characters $\delta_{\Hh^{\opp}} = \delta_{\Hh}$, and so
$$\kappa_{A}(x) = \sigma_{-i/2}(x_{(2)})\delta^{-1/2}(x_{(1)}) \delta^{-1/2}(x_{(3)}) = \delta^{-3/4}(x_{(1)})x_{(2)} \delta^{-3/4}(x_{(3)}), \quad x \in \mathcal{O}(\mathbb{H}).$$
A calculation then shows that $$\kappa_{A}(U_\pi(\xi, \eta))=  U_\pi(\delta^{-3/4}\xi, \delta^{-3/4}\eta), \quad \xi, \eta \in \mathcal{H}_\pi.$$
Therefore, the $*$-character
$$\epsilon_A: \operatorname{Fus}[\mathbb{H}]\cong \mathcal{O}(\mathbb{H}) \overset{\G}{\square} \overline{\mathcal{O}(\mathbb{H})} \to \mathbb{C}$$
is given by
\begin{align*}
    \epsilon_A(\chi(\pi)) &= \sum_{j,k} \Phi_{\mathbb{H}}(U_{\overline{\pi}}(\overline{e_j^\pi}, \overline{e_k^\pi})U_{\pi}(\delta^{-3/4}e_j^{\pi},\delta^{-3/4}e_k^{\pi}))\\
    &= \sum_{j,k} \Phi_{\mathbb{H}}(U_\pi(\delta^{-1/4}e_j^\pi, \delta^{1/4}e_k^\pi)^* U_{\pi}(\delta^{-3/4}e_j^{\pi},\delta^{-3/4}e_k^{\pi}))\\
    &= \frac{1}{\dim_q(\pi)}\sum_{j,k} \langle e_j^\pi, \delta^{-1/2}e_j^\pi\rangle \langle e_k^\pi, \delta^{-1/2} e_k^\pi\rangle = \dim_q(\pi).
\end{align*}

(2) Note that 
\begin{align*}
    &\kappa_{L^\infty(\G)}(x\otimes y) = \sigma_{-i/2}(x_{(2)})\otimes \sigma_{-i/2}(y_{(1)}) (\delta^{-1/2}(x_{(1)})\delta^{-1/2}(y_{(2)})), \quad x,y \in \mathcal{O}(\G).
\end{align*}
Therefore, if $\pi \in \operatorname{Irr}(\mathbb{H})$, we see that
$$\kappa_{L^\infty(\G)}(U_\pi(\xi, \eta)\otimes U_\pi(\xi', \eta')) = U_\pi(\delta^{-3/4}\xi, \delta^{-1/4}\eta) \otimes U_\pi(\delta^{-1/4}\xi', \delta^{-3/4}\eta'), \quad \xi, \xi', \eta, \eta'\in \mathcal{H}_\pi.$$
Consequently, by the orthogonality relations,
\begin{align*}
    &(\Phi_{\mathbb{H}}\odot \Phi_{\mathbb{H}})((U_\pi(\delta^{-1/4}e_j^\pi, e_s^\pi)^*\otimes U_\pi(e_t^\pi, \delta^{1/4}e_k^\pi)^*)\kappa_{L^\infty(\G)}(U_\pi(e_j^\pi, e_l^\pi)\otimes U_\pi(e_m^\pi, e_k^\pi)))\\
    &= \frac{1}{\dim_q(\pi)^2} \langle e_j^\pi, \delta^{-1/2} e_j^\pi\rangle \langle e_s^\pi, \delta^{-1/4} e_l^\pi\rangle \langle e_m^\pi, \delta^{1/4}e_t^\pi\rangle \langle e_k^\pi, \delta^{-1/2}e_k^\pi\rangle.
\end{align*}
We therefore deduce that
\begin{align*}
    \theta_S(\chi(\pi)) &\cong \theta_S\left(\sum_{j,k=1}^{n_\pi} U_\pi(e_j^\pi, e_k^\pi)\otimes \overline{U_\pi(\delta^{-1/4}e_j^\pi, \delta^{1/4}e_k^\pi)}\right)\\
    &= \sum_{l,m,s,t}\pi_{\mathbb{H}}(U_\pi(e_l^\pi, e_m^\pi)) \rho_{\mathbb{H}}(U_\pi(e_s^\pi, e_t^\pi)^*) \langle e_s^\pi, \delta^{-1/4} e_l^\pi\rangle \langle e_m^\pi, \delta^{1/4}e_t^\pi\rangle\\
    &= \sum_{l,m} \pi_{\mathbb{H}}(U_\pi(e_l^\pi, e_m^\pi)) \rho_{\mathbb{H}}(U_\pi(\delta^{-1/4}e_l^\pi, \delta^{1/4}e_m^\pi)^*)\\
    &= \sum_{l,m}\pi_{\mathbb{H}}(U_\pi(e_l^\pi, e_m^\pi)) \rho_{\mathbb{H}}(R(U_\pi(e_m^\pi, e_l^\pi))).
\end{align*}
\end{proof}

\begin{Theorem}\label{TheoMain2} 
    The action $\alpha: L^\infty(\mathbb{H})\curvearrowleft \G$ is (strongly) $\G$-amenable if and only if $\mathbb{H}$ is of Kac type.
\end{Theorem}
\begin{proof} Suppose that the action $\alpha: L^\infty(\mathbb{H})\curvearrowleft \G$ is amenable. Then the $*$-representation
$\theta_S: C_u^*(\operatorname{Fus}[\mathbb{H}])\to B(L^2(\mathbb{H}))$
weakly contains the $*$-representation
$\epsilon_A: C_u^*(\operatorname{Fus}[\mathbb{H}])\to \mathbb{C}$
so in particular
\begin{equation}\label{counitestimate}
    |\epsilon_A(x)|\le \|\theta_S(x)\|, \quad x \in \operatorname{Fus}[\mathbb{H}].
\end{equation}
By Proposition \ref{semi}, the $*$-representation
    $\theta_S: \operatorname{Fus}[\mathbb{H}]\to B(L^2(\mathbb{H}))$
    arises as the restriction of a $*$-representation $\mathcal{O}(\mathbb{H})\to B(L^2(\mathbb{H}))$. Therefore, \eqref{counitestimate} shows that 
    $$|\epsilon_A(x)|\le \|x\|_u, \quad x \in \operatorname{Fus}[\mathbb{H}]$$
    with $\|\cdot\|_u$ the universal norm on $\mathcal{O}(\mathbb{H})$. In particular, if $\pi\in \operatorname{Irr}(\mathbb{H})$, taking $x= \chi(\pi)\in \operatorname{Fus}[\mathbb{H}]$ yields
    $$\dim_q(\pi)\le \|\chi(\pi)\|_u \le n_\pi\le \dim_q(\pi),$$
    so that $\mathbb{H}$ is of Kac type (see e.g.\ \cite{NT14}).
\end{proof}

\subsection{Second strategy} \label{strategy2} Let $\G$ be a compact quantum group and $B= L^\infty(\mathbb{X}\backslash \G)$ an injective von Neumann algebra such that $\Delta(B)\subseteq B \bar{\otimes} L^\infty(\G)$ (i.e. $B$ is a right coideal-von Neumann algebra in $L^\infty(\G)$). Suppose that the  action $\Delta: B \curvearrowleft \G$ is amenable. By \cite{DR24}*{Proposition 4.3}, the $\G$-$W^*$-algebra $B$ is $\G$-injective and thus $\mathcal{R}(B)= C_r(\mathbb{X}\backslash \G)$ is $\G$-$C^*$-injective \cite{DH24}*{Lemma 3.17}. Then \cite{AK24}*{Proposition 5.8} implies that $\mathbb{X}$ is a compact quasi-subgroup of $\G$. Consequently, if $\mathbb{X}$ is not a compact quasi-subgroup of $\G$, the action $\Delta: B \curvearrowleft \G$ is not amenable.

As an immediate application of the above strategy, we then find:
\begin{Prop}\label{Podlesamenable}
    Consider a non-standard Podleś sphere $\mathbb{X}\backslash SU_q(2)$ for $0 < q < 1$. Then the action $L^\infty(\mathbb{X}\backslash SU_q(2))\curvearrowleft SU_q(2)$ is not amenable.
\end{Prop}
\begin{proof}
    It is well-known that $\mathbb{X}$ is not a compact quasi-subgroup of $SU_q(2)$ (see e.g.\ \cite{FST13}*{Theorem 5.1}) and that $L^\infty(\mathbb{X}\backslash SU_q(2))$ is injective (\cite{Boc95}*{Corr. 23}). Hence the action $L^\infty(\mathbb{X}\backslash SU_q(2))\curvearrowleft SU_q(2)$ is not amenable. 
\end{proof}

This provides another example of a co-amenable compact quantum group acting non-amenably on a von Neumann algebra.

\section{Outlook}

In this paper, we have shown that amenability and strong amenability coincide for compact/discrete quantum group actions on von Neumann algebras. We have then produced the first examples of amenable discrete quantum groups acting non-amenably on a von Neumann algebra.

We note that in \cite{AV23}, a notion of co-amenability was introduced for coideals of compact quantum groups (=\emph{embeddable} ergodic coactions), and of amenability for their dual discrete coideals. The connection to the amenability conditions studied in this paper remains to be explored. 

We also note that in the ergodic case, direct connections can be made with the theory of tensor C$^*$-categories. Indeed, in this case ergodic actions of $\G$ on von Neumann algebras are in one-to-one correspondence (up to appropriate isomorphism) with cyclic semisimple module W$^*$-categories $\mcD$ for $\mcC = \Rep_u^{\fd}(\G)$ \cite{DCY13}. The $*$-algebra $\mcA \rtimes \G \ltimes \overline{\mcA}$ should then, up to Morita equivalence, be definable directly from $\mcC\curvearrowright \mcD$ as a `relative tube $*$-algebra' $\mcT(\mcC \curvearrowright \mcD)$. This would lead to a relative version of the theory developed in \cites{PV15,GJ16,NY16}, with the tube $*$-algebra loc.\ cit.\ arising as $\mcT(\mcC\boxtimes \mcC^{\opp} \curvearrowright \mcC)$. Details for this will appear elsewhere. 

\emph{Acknowledgements}: The work of K. De Commer was supported by Fonds voor Wetenschappelijk Onderzoek (FWO), grant G032919N. The work of J. De Ro is supported by the FWO Aspirant-fellowship, grant 1162524N. The authors thank A. Skalski for bringing \cite{BMO20} to their attention at an early part of this project. They also thank S. Vaes and G. Pisier for making \cite{Haa93} available to them, and S. Vaes and M. Yamashita for useful feedback about the contents of the paper. We thank the referee for useful comments improving the readability of the paper.


\begin{thebibliography}{00}
\bibitem[AD79]{AD79} C. Anantharaman-Delaroche, Action moyennable d’un groupe localement compact sur une algèbre de von Neumann, \emph{Math. Scand.} \textbf{45} (1979), 289--304.
\bibitem[AD82]{AD82} C. Anantharaman-Delaroche, Action moyennable d'un groupe localement compact sur une alg\`ebre de von {N}eumann. {II}, \emph{Math. Scand.} \textbf{50} (1982), 251--268.
\bibitem[AD93]{AD93} C. Anantharaman-Delaroche, Atomic correspondences, \emph{Indiana Univ. Math. J.} \textbf{42} (2) (1993), 505--531.
\bibitem[AK24]{AK24} B. Anderson-Sackaney and F. Khosravi, Topological Boundaries of Representations and Coideals, \emph{Adv. Math.} \textbf{425} (2024), 109830.
\bibitem[AV23]{AV23} B. Anderson-Sackaney and L. Vainerman, Fusion Modules and Amenability of Coideals of Compact and Discrete Quantum Groups, \emph{preprint}, arXiv:2308.01656.
\bibitem[AS21]{AS21} Y. Arano and A. Skalski, On the Baum--Connes conjecture for discrete quantum groups with torsion and the quantum Rosenberg Conjecture, \emph{Proc. Amer. Math. Soc.} \textbf{149} (2021), 5237--5254.
\bibitem[BMO20]{BMO20} J. Bannon, A. Marrakchi and N. Ozawa, Full Factors and Co-amenable Inclusions, \emph{Comm. Math. Phys.} \textbf{378} (2020), 1107--1121.
\bibitem[BC22]{BC22} A. Bearden and J. Crann, Amenable dynamical systems over locally compact groups, \emph{Ergod. Th. \& Dynam. Sys.} \textbf{42} (2022), 2468--2508.
\bibitem[Boc95]{Boc95} F. Boca, Ergodic actions of compact matrix pseudogroups on C$^*$-algebras, In: \emph{Recent advances in operator algebras (Orl\'{e}ans, 1992)}, \emph{Ast\'{e}risque} \textbf{232} (1995), 93--109.
\bibitem[BEW21]{BEW21} A. Buss, S. Echterhoff and R. Willett, Amenability and weak containment for actions of locally compact groups on C$^*$-algebras, \emph{Preprint} (2021), to appear in \emph{Mem. Amer. Math. Soc.}
\bibitem[Cr19]{Cr19} J. Crann, Inner amenability and approximation properties of locally compact quantum groups, \emph{Indiana Univ. Math. J.} \textbf{68} (2019), 1721--1766.
\bibitem[DCDR24]{DCDR23} K. De Commer and J. De Ro, Approximation properties for dynamical W$^*$-correspondences, \emph{Adv. Math.} \textbf{458} (2024), 109958.
\bibitem[DCY13]{DCY13} 
K. De Commer and M. Yamashita. Tannaka–Krein Duality for Compact Quantum Homogeneous Spaces. I. General theory, \emph{Theory and Applications of Categories} \textbf{28} (31) (2013), 1099--1138.
\bibitem[DH24]{DH24} J. De Ro and L. Hataishi, Actions of compact and discrete quantum groups on operator systems, \emph{Int. Math. Res. Not.} \textbf{15} (2024), 11190–11220.
\bibitem[DR24]{DR24} J. De Ro, Equivariant injectivity of crossed products, \emph{preprint}, arXiv:2312.10738v2.
\bibitem[FST13]{FST13} U. Franz, A. Skalski and R. Tomatsu, Idempotent states on compact quantum groups and their classification on $U_q(2)$, $SU_q(2)$ and $SO_q(3)$, \emph{J. Noncommut. Geom.} \textbf{7} (1) (2013), 221--254.
\bibitem[GJ16]{GJ16} S. K. Ghosh and C. Jones, Annular representation theory for rigid C$^*$-tensor categories, \emph{J. Funct. Anal.} \textbf{270} (4) (2016), 1537--1584.
\bibitem[Haa79]{Haa79} U. Haagerup, $L_p$ spaces associated with an arbitrary von Neumann algebra, \emph{Alg\`{e}bres
d’op\'{e}rateurs et leurs applications en physique math\'{e}matique}, CNRS (1979), 175--184.
\bibitem[Haa93]{Haa93} U. Haagerup, Selfpolar forms, conditional expectations and the weak expectation property for C$^*$-algebras, \emph{Preprint} (1993).
\bibitem[JP10]{JP10} M. Junge and J. Parcet, Mixed-norm inequalities and operator space $L_p$ embedding theory, \emph{Mem. Am. Math. Soc.} \textbf{203} (953) (2010), vi+155p.
\bibitem[Moa18]{Moa18} M. Moakhar, Amenable actions of discrete quantum groups on von Neumann algebras. arXiv: 1803.04828
\bibitem[NT14]{NT14} S. Neshveyev and L. Tuset, Compact Quantum Groups and Their Representation Categories, Cours Sp\'{e}cialis\'{e}s 20, Soci\'{e}t\'{e} Math\'{e}matique de France (2014), 169pp.
\bibitem[NY16]{NY16} S. Neshveyev and M. Yamashita, Drinfeld centre and representation theory for monoidal categories, \emph{Comm. Math. Phys.} \textbf{345} (1) (2016), 385--434.
\bibitem[OS21]{OS21} N. Ozawa and Y. Suzuki, On characterizations of amenable C$^*$-dynamical systems and new examples, \emph{Selecta Math. (N.S.)} \textbf{72} (92) (2021), 29 pages.
\bibitem[Pis95]{Pis95} G. Pisier, Projections from a von Neumann algebra onto a subalgebra, \emph{Bull. Soc. Math. France} \textbf{123} (1995), 139--153.
\bibitem[PV15]{PV15} S. Popa and S. Vaes, Representation theory for subfactors, $\lambda$-lattices and C$^*$-tensor categories, \emph{Comm. Math. Phys.} \textbf{340} (2015), 1239--1280.
\bibitem[RW98]{RW98} I. Raeburn and D.P. Williams, Morita equivalence and continuous-trace {$C^*$}-algebras, \emph{AMS Mathematical Surveys and Monographs} \textbf{60} (1998). 
\bibitem[Ter81]{Ter81} M. Terp, $L_p$ spaces associated with von Neumann algebras, \emph{Math. Institute Copenhagen University} (1981).
\bibitem[Tom06]{Tom06} R. Tomatsu, Amenable discrete quantum groups, \emph{J. Math. Soc. Japan} \textbf{58} (4) (2006), 949--964. 
\bibitem[Vae01]{Vae01} S. Vaes, The unitary implementation of a locally compact quantum group action, \emph{J. Funct. Anal.} \textbf{180} (2001), 426--480. 
\bibitem[Wor74]{Wor74} S. L. Woronowicz, Selfpolar forms and their applications to the C$^*$-algebra theory, \emph{Rep. Math. Phys.} \textbf{6} (1974), 487--495.
\end{thebibliography}
\end{document}